\documentclass[11pt,epsfig]{article}

\usepackage{amsfonts}
\usepackage{mathrsfs}
\usepackage{amsmath}
\usepackage{epsfig}
\usepackage{amssymb}
\usepackage{graphicx}
\usepackage{amsthm}
\usepackage{dsfont}
\usepackage{amssymb}
\usepackage{float}
\usepackage{setspace}

\newtheorem{thm}{Theorem}[section]
\newtheorem{lem}{Lemma}[section]

\newtheorem{rem}{Remark}[section]
\newtheorem{exm}{Example}[section]

\date{}

\begin{document}

\title{Domain Decomposition Preconditioners for a Discontinuous Galerkin Formulation of a Multiscale Elliptic Problem}

\author{Yunfei Ma, Petter Bj{\o}rstad, Talal Rahman, and Xuejun Xu }


\maketitle

\begin{abstract} In this paper, we propose a domain decomposition method for
multiscale second order elliptic partial differential equations with highly varying coefficients.
The method is based on a discontinuous Galerkin formulation.
We present both a nonoverlapping and an overlapping version of the method.
We prove that the condition number bound of the preconditioned algebraic system
in either case can be made independent of the coefficients under certain assumptions. Also, in our analysis, we do not need to assume that the coefficients are continuous across the coarse grid boundaries. The analysis and the condition number
bounds are new, and contribute towards further extension of the theory for the discontinuous
Galerkin discretization for multiscale problems.
\end{abstract}

\section{ Introduction}
Subsurface flows in heterogenous media \cite{I.P.,I.P.b}, that is, where
the heterogeneity varies over a wide range of scales, are examples of multiscale problems.
Numerical solutions of such problems are often affected by the heterogeneity,
in particular when it is highly varying, that is, the permeability of the media can span a large scale.
In this paper we consider the numerical solution of flow problems
governed by elliptic equations with highly varying coefficients.
Domain decomposition methods use both fine scale and coarse scale
subproblems as well as interpolation operators from the subspaces to the
solution space to construct preconditioners for the original problem.
The convergence property is linked to proper treatment of the jumps in the coefficients \cite{I.P.R.}.
A key point when building such domain decomposition methods is to find
a good coarse problem that can capture relevant small scale information from the fine level.
In general,
due to the complex geometry of the conductivity field, where high and low conductivity regions often appear as
small inclusions inside subdomains and long channels across
subdomain boundaries, it is rather difficult to design a robust
domain decomposition method for such problems.\\
\indent\,When considering high conductivity regions (inclusions or channels) in
a low conductivity background, the authors of \cite{I.P.R.} used
functions that are discrete harmonic in each coarse grid block with
special boundary values as coarse basis functions (multiscale basis functions).
They introduced two indicators, $\pi(\alpha)$ and $\gamma(\alpha)$, where $\alpha$ is the coefficient representing the conductivity, to capture the effect of the jumps in $\alpha$ on the condition number.
The indicator $\pi(\alpha)$ measures how well the coarse partitioning of the whole domain $\Omega$ is,
as for instance, $\pi(\alpha)$ behaves well as long as the high conductivity regions (inclusions) do not cross
any subdomain boundaries, and badly otherwise.
The indicator $\gamma(\alpha)$ is a measure for the weighted energy norm of the coarse basis functions, and it depends on the choice of boundary conditions, as for instance,
an oscillatory  boundary condition is often needed in order to keep the energy norm low, consequently for
$\gamma(\alpha)$ to behave well.
  The authors in \cite{I.P.R.} used the traditional $L^2$ projection operator from the fine space to the coarse space,
and proved that the condition number bound is independent of the jumps.
This approach assumes that the conductivity coefficient is continuous
across coarse grid boundaries, although their numerical results did not seem to require this.
Recently, in \cite{G.E.}, the authors proposed an overlapping domain decomposition method for the multiscale problem.
The idea of their method is based on the fact that high conductivity
regions correspond to the smallest eigenvalues of the system.
Consequently, they used the corresponding eigenfunctions as coarse basis functions,
and proved a weighted Poincar\'{e} inequality resulting in a condition number bound independent of the jumps.
 They also proposed an overlapping domain decomposition methods for the Schur complement system.\\
\indent Discontinuous Galerkin methods may offer several important and valuable computational advantages
over their conforming Galerkin counterparts.
The finite element spaces are not subject to inter-element continuity conditions,
and the local element spaces can be defined independently.
This makes discontinuous Galerkin methods well suited for their applications to multiscale problems
with piecewise constant coefficients relative to a fine triangulation.
A domain decomposition method for the discontinuous Galerkin formulation of
a multiscale elliptic problem has recently been proposed in \cite{D.S.}. There a composite discontinuous Galerkin formulation,
that is, a regular continuous formulation inside each subdomain and a discontinuous Galerkin formulation
across the subdomain boundaries, has been used.
There the coarse space consists of piecewise constant basis functions over the coarse partition.
The condition number bound of this method is shown to be
\begin{align}
\max_i\max_j\frac{\overline{\alpha}_i}{\underline{\alpha}_i}\frac{H_i}{h_ih_{ij}}, \label{eq:cond1}
\end{align}
 where $\overline{\alpha}_i$ and $\underline{\alpha}_i$ represent the
maximum and the minimum of the coefficients inside a boundary layer of the subdomain $\Omega_i,$ $H_i$ is the diameter of $\Omega_i,$ $h_i$ is
the fine mesh size in $\Omega_i,$ and $h_{ij}=2h_ih_j/(h_i+h_j)$, the harmonic average of $h_i$ and $h_j$.
We note that the condition number bound above depends
on the jump of the coefficients inside the boundary layer of each subdomain.  \\
\indent The present work is an extension of the work in \cite{I.P.R.} to a discontinuous Galerkin formulation.
We use the same bilinear form as the one given in \cite{Z.X.S.} with harmonic average weight
functions, defined on the fine space which is the space of piecewise linear polynomials with respect to the fine triangulation.
A composite discontinuous Galerkin formulation, c.f., \cite{D.S.}, is used on the coarse space
whose basis functions are the multiscale basis functions given
by the oscillatory boundary conditions, c.f., \cite{I.P.R.}.
We present both a nonoverlapping and an overlapping method.\\
\indent A new indicator $\beta(\alpha)$ is introduced measuring the $L^2$ norm of the jump
of the coarse basis functions across coarse grid boundaries (see (\ref{eq:beta})).
We show that, under certain assumptions, slightly weaker than in \cite{I.P.R.}, our methods are robust.
For the nonoverlapping case, we show that the condition number bound is
\begin{align}
\eta\max_{\partial K_{mk}\atop m>k}\max_{e_{jl}\subset \partial K_{mk}\atop j>l}W_e\gamma(1)\frac{H}{h}+\max\big{(}\gamma(\alpha), \eta\beta(\alpha)\big{)}, \label{eq:cond2}
\end{align}

where $W_e$ is the harmonic average weight function (see (\ref{eq:weight})), $\gamma(\alpha)$ and $\beta(\alpha)$ are the
indicators, and $\eta$ is a penalty parameter.
The numerical experiments support the assumption that this bound is sharp.
However, due to the presence of $W_e$ in the bound, the method may
have a large condition number when high conductivity channels cross subdomain boundaries.
Note that, this is the case for (\ref{eq:cond1}) even when there are no channels crossing subdomain boundaries.
The situation gets improved in the overlapping case.
We show that, for the overlapping method, the condition number bound is
\begin{align}
\pi(\alpha)\max\big{(}\gamma(1),\beta(1)\big{)}\max_i\frac{H_i}{\delta_i}+\max\big{(}\gamma(\alpha), \eta\beta(\alpha)\big{)}. \label{eq:cond3}
\end{align}
\indent The rest of the paper is organized as the following. In Section 2, we introduce our model problem and the discontinuous Galerkin discrete formulation.
The two level additive Schwarz domain decomposition methods, first the
nonoverlapping and then the overlapping version, are defined and analyzed in Section 3.
Section 4 is devoted to the numerical results.

\section{ Problem Setting}
 \setcounter{equation}{0}
 Throughout this paper, we
adopt standard notations from the Lebesgue and Sobolev space theory
(see \cite{Adams}).
We further use $A
\lesssim B$ to denote $A \leq CB$ with a positive constant $C$ depending
only on the shape regularity of the meshes, and $A \approx B$
to denote $A \lesssim B \lesssim A$.

\indent Consider the self adjoint elliptic problem on a polygonal domain $\Omega$ with boundary $\partial\Omega$,
\addtocounter{equation}{0}
\begin{equation}
\-{L}(u)=-\nabla\cdot(\alpha(x)\nabla u)=f,~~~x\in\Omega,
\label{eq:eq}
\end{equation}
\begin{equation}
 u=g,~~~ x \in \partial\Omega,
\label{eq:bc}
\end{equation}
where the coefficient $\alpha(x)\in L^{\infty}(\Omega)$ with $\alpha(x)\geq \alpha_0>0,$ representing the conductivity field on $\Omega,$ is
piecewise constant on the fine mesh $\mathcal{T}_h,$
$\alpha(x)$ may have (discontinuous) jumps across the elements. Also we let $f\in L^2(\Omega).$\\
 \indent\,\,First, we introduce some notations.
   $\mathcal{T}_h$ denotes a fine triangulation of the whole domain $\Omega$ which is
quasi-uniform, $\mathcal{T}_h=\bigcup_{\tau\in \Omega}\tau, $ where $\tau$ represents a
small triangle in $\mathcal{T}_h$ and $h$ is the mesh size.
  $\mathcal{T}_H$ denotes a coarse triangulation which we get by partitioning $\Omega$ into triangular substructures $\mathcal{T}_H$,
 which is also quasi-uniform with the mesh size $H$.
We assume that the boundary edge of each coarse grid element in $\mathcal{T}_H,$ is aligned
with the edges of the elements in the fine triangulation $\mathcal{T}_h$.
 We denote the elements of the coarse and the fine triangulation by $K$, and $\tau$, respectively.
$e$ denotes an edge of a fine element $\tau$, and $\mathcal{E}$ is the union of all
edges in $\Omega$, $\mathcal{E}=\bigcup_{e\subset \partial\tau}e.$
Additionally, $\mathcal{E}=\mathcal{E}_I\bigcup\mathcal{E}_D$, where $\mathcal{E}_I:=\bigcup_{j>l}e_{jl}$ refers to all the inner edges
and $\mathcal{E}_D:=\bigcup_{e\subset \partial \Omega}e$ refers to all the edges touching $\partial\Omega.$
Given a coarse triangulation $\mathcal{T}_H=\{K_i\}_{i=1}^N,$ we let
$\partial K_i$ be the boundary of the element $K_i$, and $\partial K_{ij}=\partial K_i\cap \partial K_j$
be the open edge shared by the elements $K_i$ and $K_j$.

 \indent\,\,Next, we introduce two weight functions related to each fine edge $e\in \mathcal{E}_I$.
We first denote the two fine elements sharing an edge $e$ by $\tau_+^e$ and $\tau_-^e$, and
denote the coefficients of the two elements by $\alpha_+^e$ and $\alpha_-^e$ respectively.
The weight functions $w_+^e$ and $w_-^e,$ associated with the edge $e,$ are defined as
$$w_+^e+w_-^e=1,$$ where $w_+^e=\alpha_-^e/(\alpha_+^e+\alpha_-^e)$ and $w_-^e=\alpha_+^e/(\alpha_+^e+\alpha_-^e).$
The harmonic average weight function, associated with the edge $e$, is then
\begin{align}
 W_e=w_+^e\alpha_+^e+w_-^e\alpha_-^e=\frac{2\alpha_-^e\alpha_+^e}{\alpha_-^e+\alpha_+^e}. \label{eq:weight}
\end{align}
The following inequalities hold,
\begin{align}
\frac{(w_+^e)^2\alpha_+^e}{W_e}=
\frac{(w_+^e)^2\alpha_+^e}{w_+^e\alpha_+^e+w_-^e\alpha_-^e}\leq\frac{(w
_+^e)^2\alpha_+^e}{w_+^e\alpha_+^e}
\leq w_+^e\leq 1 , \label{eq:weight-ineq}
\end{align}
and
$$\frac{(w_-^e)^2\alpha_-^e}{W_e}\leq 1.$$
\indent Now, let $\alpha_{min}=\min(\alpha_-^e, \alpha_+^e)$ and $\alpha_{max}=\max(\alpha_-^e, \alpha_+^e)$, it then follows immediately that
$\alpha_{min}\leq W_e \leq 2\alpha_{min}.$
In the above definitions of $W_e$ and $w^e$, for $e\in\mathcal{E}_D,$
we set $w_e=1$ and $W_e=\alpha^e.$\\
\indent Let the jump across an edge $e$ be
\begin{align}
[v]=
\begin{cases}v_+^e\mathbf{n}_+^e+v_-^e\mathbf{n}_-^e,\quad \ \ &
e\in\mathcal{E}_I,\\
v_e\mathbf{n}_e,\quad \ \ & e\in\mathcal{E}_D,\\
\end{cases}
\label{eq:jump}
\end{align}
where $\mathbf{n}_+^e$ and $\mathbf{n}_-^e$ denote the unit outward normal
vectors of $\tau_+^e$ and $\tau_-^e,$ respectively.
The weighted average $\{v\}_w^e$ is defined similarly by
\begin{align}
\{v\}_w^e=
\begin{cases}w_+^ev_+^e+w_-^ev_-^e,\quad \ \ &
e\in\mathcal{E}_I,\\
w_ev_e,\quad \ \ & e\in\mathcal{E}_D.\\
\end{cases}
\label{eq:average}
\end{align}

For the ease of our presentation, we define all our norms and function spaces here.

Given a domain $D\subset \Omega,$
we define the standard $L^\infty$ norm on $D$ as
$$||u||_{0,\infty,D}^2:={\rm ess}\,\, {\rm sup}\{|u(x)|:x\in D\},$$
and the standard $L^2$ norm on $D$ as $$||u||_{0,D}^2:=\int_D u^2dx.$$
A weighted $H^1$
seminorm on $D$ is defined by $$|u|_{1,\alpha,D}^2:=\int_D\alpha\nabla u\cdot\nabla u \, dx,$$ where $\alpha(x)>0,\,\,\forall x\in D.$
With $\alpha(x)=1,\,\,\forall x\in D,$ we get the standard $H^1$ seminorm $$|u|_{1,D}^2:=\int_D\nabla u\cdot\nabla u \, dx.$$
A weighted norm, based on the
discontinuous Galerkin formulation, on $D$, is defined by $$\|u\|_{1,h,\alpha, D}^2:=\sum_{\tau\subset D}\int_{\tau}\alpha\nabla u\cdot
\nabla u dx+\sum_{e_{jl}\subset D \atop j>l}
\frac{W_e}{h_e}\|[u]\|_{0,e}^2+\sum_{e\subset \partial D}
\frac{W_e}{h_e}\|u_e\|_{0,e}^2,$$ where $\bigcup_{j>l}e_{jl}$ refers to all the fine element edges in the interior of $D.$ We use $$|u|_{1,h,\alpha, D}^2:=\sum_{\tau\subset D}\int_{\tau}\alpha\nabla u\cdot
\nabla u dx+\sum_{e_{jl}\subset D \atop j>l}
\frac{W_e}{h_e}\|[u]\|_{0,e}^2$$ to denote the corresponding semi norm on $D$. If $\alpha(x)=1,\,\,\forall x\in D,$ we get the following norm
$$\|u\|_{1,h,D}^2:=\sum_{\tau\subset D}\int_{\tau}\nabla u\cdot
\nabla u dx+\sum_{e_{jl}\subset D \atop j>l}
\frac{1}{h_e}\|[u]\|_{0,e}^2+\sum_{e\subset \partial D}
\frac{1}{h_e}\|u_e\|_{0,e}^2,$$ and seminorm
$$|u|_{1,h, D}^2:=\sum_{\tau\subset D}\int_{\tau}\nabla u\cdot
\nabla u dx+\sum_{e_{jl}\subset D \atop j>l}
\frac{1}{h_e}\|[u]\|_{0,e}^2.$$
Let $H^s(\mathcal{T}_h )$ be a broken Sobolev space of degree $s>0$ defined as
$$H^s(\mathcal{T}_h)=\{v\in L^2(\Omega):v|_{\tau}\in H^s(\tau),\,\,\forall \tau\in\mathcal{T}_h\},$$
and $V^s(\mathcal{T}_h)$ be its subspace defined as
$$V^s(\mathcal{T}_h)=\{v\in H^s(\mathcal{T}_h):\nabla\cdot(\nabla v)\in L^2(\tau),\,\,\forall \tau\in\mathcal{T}_h\}.$$

With the above preparation, we can now define our discontinuous Galerkin bilinear form $a(u,v)$ and the right hand side $f(v)$ for the continuous problem (\ref{eq:eq}) and (\ref{eq:bc}). They are defined as follows, c.f., \cite{Z.X.S.}. For $u, v\in V^{1+\epsilon}(\mathcal{T}_h)$,
\begin{align}
\begin{split}
a(u,v)=&\sum_{\tau}\int_{\tau}\alpha\nabla u\cdot \nabla vdx-\sum_{e_{jl}\subset \Omega\atop j>l}\int_{e_{jl}}\{\alpha
\nabla u\}_w^e[v]ds-\sum_{e_{jl}\subset \Omega\atop j>l}\int_{e_{jl}}\{\alpha \nabla
v\}_w^e[u]ds\\&+\eta\sum_{e_{jl}\subset \Omega\atop j>l} \frac{W_e}{h_e}\int_{e_{jl}} [u][v]ds-\sum_{e\subset \partial\Omega}\int_ew_e\alpha_e
\nabla u_e\cdot v_e\mathbf{n}_eds\\&-\sum_{e\subset \partial\Omega}\int_{e}w_e\alpha_e \nabla
v_e\cdot u_e\mathbf{n}_eds+\eta\sum_{e\subset \partial\Omega} \frac{W_e}{h_e}\int_e u_ev_eds,
\label{eq:bf}
\end{split}
\end{align}
and
\begin{align*}
f(v)=\sum_{\tau}\int_{\tau}fvdx-\sum_{e\subset \partial\Omega}\frac{W_e}{h_e}\int_eg(\alpha\nabla v\cdot\mathbf{n}_e)ds+\eta\sum_{e\subset \partial\Omega}\frac{W_e}{h_e}\int_egvds,
\end{align*}
where $\cup_{j>l}e_{jl}$ refers to all the fine element edges in the interior of $\Omega,$ $W_e$ is the harmonic weight function defined in (\ref{eq:weight}),
and $\eta$ is a penalty parameter.\\
\indent Here and in the text below, for simplicity, we use $W_{e},$ $h_e$ and $\{\cdot\}_w^e$ instead of $W_{e_{jl}},$ $h_{e_{jl}}$ and  $\{\cdot\}_w^{e_{jl}}$ respectively if there is
no confusion. Moreover, as stated in \cite{Z.X.S.}, for $f\in L^2(\Omega),$ the flux $\mathbf{\sigma}=-\alpha\nabla u$ is in $H({\rm div}; \Omega)\cap H^\epsilon(\mathcal{T}_h)^2$,
where $H({\rm div};\Omega)=\{\mathbf{v}\in L^2(\Omega)^2:\nabla\cdot\mathbf{v}\in L^2(\Omega)\}$ is an Hilbert space equipped with the norm
$||\mathbf{v}||_{H_{{\rm div};\Omega}}=(||\mathbf{v}||_{0,\Omega}^2+||\nabla\cdot \mathbf{v}||_{0,\Omega}^2)^{1/2}.$
The integrations on the fine edge $e_{jl}$ in (\ref{eq:bf}) above can be understood in the weak sense, see \cite{Z.X.S.} for details.

Due to \cite{Z.X.S.}, the the weak solution of $(2.1)$ and $(2.2),$ $u\in H^{1+s}(\Omega)\cap V^{1+\epsilon}(\mathcal{T}_h),$ $0<\epsilon\leq s\leq1,$
satisfies the following variational equation
\begin{align}
a(u,v)=f(v),\,\,\forall v\in V^{1+\epsilon}(\mathcal{T}_h).
\label{eq:cp}
\end{align}
\indent Next, we define the finite element space $V^h$ associated with $\mathcal{T}_h$ as follows.
For any fine triangle $\tau$ let $P_{1}(\tau)$ denote the set of all
linear polynomials on $\tau$. The finite element space $V^h$
is then defined as
$$V^h = \{v: v|_{\tau} \in P_1(\tau), \, \forall \tau \in \mathcal{T}_h \}. $$
The bilinear form on the finite element space $V^h$ is defined as $a_h(u,v)=a(u,v),$ $\forall u,v\in V^h.$

We can now formulate our discrete problem: Find $u_h \in V^h$ such that
\begin{equation}
 a_{h}(u_h,v_h) = f(v_h),\,\,\forall v_h \in V^h.    \label{eq:dp}
\end{equation}

Naturally, the above bilinear form induces a norm in the space $V_h$, which is
$$\|v\|_{1,h,\alpha,\Omega}^2=\sum_{\tau\subset\Omega}\int_{\tau}\alpha\nabla v\cdot
\nabla vdx+\sum_{e_{jl}\in\subset\Omega\atop j>l}
\frac{W_e}{h_e}\|[v]\|_{0,e}^2+\sum_{e\subset \partial\Omega}
\frac{W_e}{h_e}\|v_e\|_{0,e}^2.$$

\indent The next lemma gives us the continuity and coercivity of the bilinear form $a_{h}(\cdot,\cdot).$
\begin{lem} There exists a constant $C,$ such that \\
{\rm (i)} \hspace*{3.0cm} $|a_{h}(u,v)|\leq C||u||_{1,h,\alpha,\Omega}^2||v||_{1,h,\alpha,\Omega}^2, \,\,\forall u,v\in V^h.$ \\[\baselineskip]
{\rm (ii)}There exists positive constant $\eta_0>0,$ such that for all $\eta>\eta_0,$
$$a_{h}(u,u)\geq C(\eta)||u||_{1,h,\alpha,\Omega}^2, \,\,\forall u\in V^h, $$ where $C(\eta)$ is independent of the jump in the coefficient.
\label{lemma:coer}
\end{lem}
As a consequence of the above lemma, (\ref{eq:dp}) has a unique solution. The proof of Lemma \ref{lemma:coer} as well as an error estimate is given in \cite{Z.X.S.}.\\

\section{ The Schwarz methods}
\setcounter{equation}{0}
In this section, we use the Schwarz framework \cite{S.B.G.,T.W.} to design and analyze
our additive Schwarz domain decomposition methods for the discontinuous Galerkin
formulation. Let $\Omega$ be partitioned
into a family of nonoverlapping open subdomains $\{\Omega_i,1\leq i\leq N\}$
with $$\overline{\Omega}=\bigcup_{i=1}^N\overline{\Omega_i},\,\,\Omega_i\cap\Omega_j=\emptyset,\,\,i\neq j.$$
 Next, we give the overlapping
partition $\{\Omega_i'\}_{i=1}^N$ by extending each subregion $\Omega_i$ into a larger region
$\Omega_i'$,
 i.e., $\Omega_i\subset \Omega_i'$, so that $$\overline{\Omega}=\bigcup_{i=1}^N \Omega_i'.$$
We assume that both the nonoverlapping and overlapping partition are aligned with $\mathcal{T}_h$ and $\mathcal{T}_H$.
We denote by $\delta$ the minimum of the distance between the boundaries
of $\Omega_i$ and $\Omega_i'$, i.e., $\delta=\min\,\, {\rm dist}(\partial\Omega_i'\backslash\partial\Omega,\partial\Omega_i\backslash \partial\Omega)$.
If there exists a constant number $C>0$ such that $\delta\geq CH$, we say
that $\{\Omega_i'\}^N_{i=1}$ has generous overlap, and if $\delta$ is proportional to $h$, we say it has small overlap.\\
\indent Following the Schwarz framework, the space $V^h$ is split into a number of local
subspaces and a global coarse space, i.e., $V^h=V_0+\sum_{i=1}^N V_i^h,$ where $V_0$ is the
coarse space and $\{V_i^h\}_{i=1}^N$ are the local subspaces.
For the coarse space, $V_0$, we use the standard
multiscale finite element basis functions as described in \cite{I.P.R.,T.X.Z.b}, where, we allow the basis functions
to be continuous inside each coarse element $K$, and discontinuous across $\partial K$.
Below is the description of $V_0.$\\
\indent We define our coarse space associated with $\mathcal{T}_H$ as follows
\begin{align}
V_0=span\{\phi_{p,K}, \,\,{\rm for\,\,all}\,\, x_p\in \mathcal{N}^H(K),\,\,\forall K\in\mathcal{T}_H \},\label{eq:cs}
\end{align}
where $\mathcal{N}^H(K)$ includes all the vertex nodes of $K$, and $\phi_{p,K}$ denotes
the multiscale basis function defined in (\ref{eq:discrete}).
To be more specific, we need to introduce suitable boundary
data $\psi_{p,\partial K}$, which is required to be
piecewise linear(w.r.t. the given fine mesh $\mathcal{T}_h$ restricted to $\partial K$), and to satisfy the following:
 $\,\,\psi_{p,\partial K}(x_p'^H)=\delta_{p,p'}, \,\,\forall \,\,x_p, x_p' \in\mathcal{N}^H(K),$
 $\,\,0\leq\psi_{p,\partial K}(x)\leq 1, \,\, {\rm and} \sum_{x_p\in \mathcal{N}^H(K)}\psi_{p,\partial K}(x)=1,\,\,\forall\,\, x\in\partial K,$ and
 $\,\,\psi_{p,\partial K}(x)=0, \,\,{\rm on\,\, the\,\, edge \,\,of} \,\,K \,\,{\rm opposite}\,\, x_p.$\\
\indent An obvious choice for the boundary data $\psi_{p,\partial K}$ satisfying the above conditions is the standard linear boundary condition.\\
\indent The linear boundary condition works well when the high conductivity regions lie strictly inside coarse grid blocks. However, when they touch
the the coarse grid boundaries, linear boundary condition fails. In this case, we use another boundary condition,
also known as oscillatory boundary condition, c.f., \cite{I.P.R.}, which satisfies the above conditions and is effective. The description of this boundary condition is as follows:\\
\indent Let $\Upsilon$ be an edge of the coarse mesh $\mathcal{T}_H$ with
end points $x_p$ and $x_{p'}$, and $\alpha^\Upsilon$ be the restriction of $\alpha$ to $\Upsilon.$
Then the oscillatory boundary condition is given by the finite element solution of the following
two-point boundary value problem:
$$-(\alpha^\Upsilon(\psi_p^\Upsilon)')'=0,\,\,x\in\Upsilon$$
$$\psi_p^\Upsilon(x_p)=1,\,\,{\rm and}\,\, \psi_p^\Upsilon(x_{p'})=0.$$
\indent Since the coefficient $\alpha^\Upsilon$ is piecewise constant, the finite element solution of the above equation can be expressed explicitly by
\begin{equation}
\psi_p^\Upsilon(x)=(\int_\Upsilon(\alpha^\Upsilon)^{-1}ds)^{-1}(\int_{\Upsilon_x}(\alpha^\Upsilon)^{-1}ds),\,\,{\rm for \,\,all} \,\,x\in\Upsilon,\label{eq:ebf}
\end{equation}
 where $\Upsilon_x$ denotes the line from $x_{p'}$ to $x.$ The function $\psi_p^\Upsilon$ is continuous and piecewise linear with respect to $\Upsilon\cap \mathcal{T}_h.$ We set $\psi_{p,\partial K}=\psi_p^\Upsilon$ on each edge $\Upsilon$ of $K,$
containing $x_p,$ and $\psi_{p,\partial K}|_\Upsilon=0$ on the edge opposite to $x_p.$\\
\indent Once the boundary condition $\psi_{p,\partial K}$ is determined, $\phi_{p,K}$ is constructed
by a discrete harmonic extension inside $K.$ First, define the $P_1$-conforming finite element space associated with $\mathcal{T}_h$ as
$$S^h(\Omega)=\{v: v|_{\tau} {\rm \,\,is \,\,linear\,\, and}\,\ v \,\,{\rm is\,\, a\,\,continous\,\,function\,\,over}\,\,\Omega\}.$$
\indent Then $\phi_{p,K}\in S^h(K)$ can be defined as
\begin{align}
\int_K \alpha\nabla \phi_{p,K}\cdot \nabla v_h=0,\,\,{\rm for\,\,all}\,\,v_h\in S_0^h(K)\,\,
       {\rm subject \,\,to}\,\, \phi_{p,K}|_{\partial K}=\psi_{p,\partial K}, \label{eq:discrete}
\end{align}
where $S_0^h(K)=S^h(K)\cap H_0^1(K).$
The most important property of $\phi_{p,K}$ is the energy minimizing, which can be stated as follows,
\begin{equation}
|\phi_{p,K}|_{1,\alpha, K}\leq|\theta|_{1,\alpha,K}, \,\,{\rm for\,\, all} \,\,\theta\in S^h(K) {\rm\,\, which\,\, satisfies }\,\,\theta|_{\partial K}=\phi_{p,K}|_{\partial K}.\label{eq:em}
\end{equation}
\indent Note that our coarse space $V_0$ actually includes functions which are required to be continuous inside each coarse element $K$ and discontinuous across coarse grid boundaries, and obviously we have $V_0\subset V^h$.\\
\indent Having the above preparations, we may define $a_{0}(\cdot, \cdot)$ the bilinear form associated with
$V_0$ as
\begin{equation}
\begin{split}
a_{0}(u,v)=&\sum_{K}\int_{K}\alpha\nabla u\cdot \nabla vdx-\sum_{\partial K_{mk}\subset\Omega\atop m>k}\sum_{e_{jl}\subset \partial K_{mk}\atop j>l}\int_{e_{jl}}\{\alpha \nabla u\}_w^e[v]ds\\&-\sum_{\partial K_{mk}\subset\Omega\atop m>k}\sum_{e_{jl}\subset \partial K_{mk}\atop j>l}\int_{e_{jl}}\{\alpha \nabla v\}_w^e[u]ds+\eta\sum_{\partial K_{mk}\subset\Omega\atop m>k}\sum_{e_{jl}\subset \partial K_{mk}\atop j>l}\int_{e_{jl}} \frac{W_e}{h_e} [u][v]ds\\&-\sum_{\partial K\subset\partial\Omega}\sum_{e\subset \partial K}\int_ew_e\alpha_e \nabla u_e\cdot v_e\mathbf{n}_eds-\sum_{\partial K\subset\partial\Omega}\sum_{e\subset \partial K}\int_ew_e\alpha_e \nabla v_e\cdot u_e\mathbf{n}_eds\\&+\eta\sum_{\partial K\subset\partial\Omega}\sum_{e\subset \partial K}\int_e \frac{W_e}{h_e} u_ev_eds, \,\,\forall u, v\in V_0.\label{eq:cbf}
\end{split}
\end{equation}
\indent Note that $V_0\subset V^h,$ and $a_{0}(u,v)=a_{h}(u,v),\,\,\forall u,v \in V_0$, hence it follows from Lemma \ref{lemma:coer} that $a_{0}(v,v)$ is coercive on $V_0$.

\indent Next, we define the local spaces, $\{V_i^h\}_{i=1}^N$ associated with subdomain partition $\{B_i\}_{i=1}^N$.
\begin{align}
V_i^h=\{v\in V^h: v=0 \,\,{\rm in} \,\,\Omega\backslash \overline{B_i}\},\,\,i=1\cdot\cdot\cdot N, \label{eq:ls}
\end{align}
where $B_i=\Omega_i$ for the nonoverlapping partition, and $B_i=\Omega_i'$ for the overlapping partition.\\
\indent The corresponding local bilinear forms $a_{i}(u,v)$ can be defined as follows:
\begin{equation}
\begin{split}
a_{i}(u,v)=&\sum_{\tau\subset B_i}\int_{\tau}\alpha\nabla
u\cdot\nabla vdx-\sum_{ e_{jl}\subset B_i\atop j>l}\int_{e_{jl}}\{\alpha
\nabla u\}_w^e[v]ds-\sum_{e_{jl}\subset B_i\atop j>l}\int_{e_{jl}}\{\alpha \nabla v\}_w^e[u]ds
\\&+\eta\sum_{e_{jl}\subset B_i\atop j>l}\int_{e_{jl}}
\frac{W_e}{h_e}[u][v]ds-\sum_{ e_{jl}\subset\Gamma_i\atop j>l}\int_{e_{jl}}w_e\alpha_e \nabla u_e\cdot v_e
\mathbf{n}_eds\\&-\sum_{ e_{jl}\subset\Gamma_i\atop j>l}\int_{e_{jl}}w_e\alpha_e \nabla v_e\cdot
u_e\mathbf{n}_eds+\eta\sum_{ e_{jl}\subset\Gamma_i\atop j>l}\int_{e_{jl}} \frac{W_e}{h_e} u_ev_eds\\&-\sum_{ e\subset\Xi_i}\int_{e}w_e\alpha_e \nabla u_e\cdot v_e
\mathbf{n}_eds-\sum_{ e\subset\Xi_i}\int_{e} w_e\alpha_e\nabla v_e\cdot
u_e\mathbf{n}_eds\\&+\eta\sum_{ e\subset\Xi_i}\int_{e} \frac{W_e}{h_e} u_ev_eds,\,\, \forall u, \,\,v\in V^h_i, \label{eq:lbf}
\end{split}
\end{equation}
where $\Gamma_i=\partial B_i\backslash \partial \Omega, \Xi_i=\partial B_i\cap\partial \Omega.$\\
\indent Note that $V_i^h\subset V^h,$ and $a_{i}(u,v)=a_{h}(u,v),\,\,\forall u,v \in V_i^h$, hence it follows from Lemma \ref{lemma:coer} that $a_{i}(v,v)$ is coercive on $V_i^h$.
\begin{rem}
{\rm
Note that the local bilinear form $a_{i}(\cdot,\cdot)$ on $V_i^h\times V_i^h$ is actually the
restriction of $a_{h}(\cdot,\cdot)$ on $V_i^h\times V_i^h.$ Thus, in the implementation of this method,
after we build the global stiffness matrix $A$ corresponding to the bilinear form $a_{h}(\cdot,\cdot),$ we
can easily get the local stiffness matrix $A_i$ corresponding to $a_{i}(\cdot,\cdot)$ by taking the $i$-$th$ diagonal block of $A.$
}
\end{rem}

\indent With this preparation, the two level additive Schwarz domain decomposition method can be presented as follows.

\indent For $1\leq i\leq N,$ we define the operators $T_i:V^h\rightarrow V_i^h$ by
$$a_{i}(T_i u,v)=a_{h}(u,v),\,\, \forall u\in V^h,\,\,v\in V_i^h,$$
and for $i=0$, we define the operator $T_0$ by
$$a_{0}(T_0 u,v)=a_{h}(u,v),\,\,\forall u\in V^h,\,\,v\in V_0.$$
Clearly, each of these problems has a unique solution.

We define the additive operator as
$$T=T_0+T_1+\cdot\cdot\cdot+T_N,$$
and replace (\ref{eq:dp}) by the operator equation
\begin{align}
Tu_h=f_h,\label{eq:ddm}
\end{align}
where $$f_h=\sum_{i=0}^Ng_i \,\, \, \mbox{with } \, g_i=T_iu_h,$$ and $u_h$ is the solution of (\ref{eq:dp}).\\

\subsection{\rm Analysis}
In this section, we estimate the condition number
of both the nonoverlapping and overlapping additive Schwarz method.
We use the standard Schwarz framework \cite{S.B.G.,T.W.}.\\
\indent  We will see that, for the nonoverlapping method, the condition number will
depend on the two indicators $\gamma(\alpha)$ and $\beta(\alpha)$, while for the overlapping method,
it depends on an additional indicator $\pi(\alpha)$ which will be defined in Section 3.1.2.\\
\indent The first indicator $\gamma(\alpha)$, which is borrowed from \cite{I.P.R.}, measures the maximum weighted energy of all the coarse basis
functions, which can be used to have an indication of how well the coarse basis functions are
constructed.
We will show that,  by choosing suitable boundary conditions, $\gamma(\alpha)$ can be bounded independently of the jumps in the coefficients.

\indent $\mathbf{Coarse\,\, robustness\,\, indicator}$ \cite{I.P.R.}.
Given a coarse triangulation $\mathcal{T}_H,$ and the set of coarse basis
functions $\{\phi_{p,K}, \,\, \forall x_p\in\mathcal{N}^H(K),\,\, \forall K\in\mathcal{T}_H,\}$ then
$$\gamma(\alpha)=\max_{K\in \mathcal{T}_H}\max_{x_p\in\mathcal{N}^H(K)}|\phi_{p,K}|_{1,\alpha, K}^2.$$
\indent Next, we introduce the indicator $\beta(\alpha),$ which
measures the weighted $L_2$ norm of the jump of the multiscale basis functions on the coarse grid boundaries. It is the maximum value of $\beta^I(\alpha),$ which corresponds to the integration on the inner coarse grid boundaries that are inside $\Omega$, and $\beta^B(\alpha),$ which corresponds to the integration on the coarse grid boundaries which intersect with $\partial\Omega.$  The term $\beta(\alpha)$ enters into our analysis due to the use of the DG bilinear form on the coarse space.

\indent $\mathbf{Multiscale\,\, DG\,\, indicator}$.
Given a coarse triangulation $\mathcal{T}_H,$ and the set of coarse basis
functions $\{\phi_{p,K}, \forall x_p \in\mathcal{N}^H(K),\,\, \forall K\in\mathcal{T}_H,\}$ then
\begin{align}
\beta^I(\alpha)= \max_{\partial K_{mk}\subset\Omega\atop m>k}\max_{x_p\in\mathcal{N}^H(K_m)\atop x_p\in\overline{ \partial K_{mk}}}
 \sum_{e_{jl}\subset\partial K_{mk}\atop j>l}(\int_{e_{jl}}\frac{W_e}{h_e}|[\phi_p]|^2ds). \label{eq:betaI}
\end{align}
where
$[\phi_p]=\phi_{p,K_m}-\phi_{p,K_k},$ $\partial K_{mk}$ is the edge shared by the coarse elements $K_m$ and $K_k,$ $W_e$ is defined in (\ref{eq:weight}), and $h_e$ is the length of the edge $e_{jl}.$\\
\indent  However, when $\partial K_m\cap \partial \Omega\neq\emptyset,$ we set
\begin{align}
\beta^B(\alpha)=\max_{\partial K_m\subset \partial\Omega}\max_{x_p\in\mathcal{N}^H(K_m)\atop x_p\in \overline{\partial K_m}}
  \sum_{e\subset\partial K_m}(\int_e\frac{W_e}{h_e}|\phi_{p,m}|^2ds). \label{eq:betaB}
\end{align}
\indent Having all the above preparations, we define
\begin{align}
\beta(\alpha)=\max\big{(}\beta^I(\alpha),\beta^B(\alpha)\big{)}. \label{eq:beta}
\end{align}

\subsubsection{\rm  Nonoverlapping additive Schwarz method}
In this section, we propose a two level additive Schwarz method with nonoverlapping subdomains, and present an analysis of the condition number. Since it is a two level additive Schwarz method, we use $V_0$ in the previous section as our coarse space with bilinear form $a_0(\cdot,\cdot)$. For the local subspaces, by taking $B_i=\Omega_i$ in (\ref{eq:ls}), we have
\begin{align*}
V_i^h=\{v\in V^h: v=0 \,\,{\rm in} \,\,\Omega\backslash \overline{\Omega_i}\},\,\,i=1\cdot\cdot\cdot N.
\end{align*}
\indent Note that, in this case, $V^h$ is a direct sum of $\{V_i^h\}.$\,\, The local
forms $a_{i}(\cdot,\cdot),\,\,1\leq i\leq N$ are given by (\ref{eq:lbf}) with $B_i=\Omega_i$.\\
\indent In order to estimate the condition number of the two level nonoverlapping domain decomposition method,
we need to define an interface bilinear form, $I(\cdot,\cdot):V^h\times V^h\rightarrow R,$ as follows: $\forall u,v \in V^h$,
\begin{align}
\begin{split}
I(u,v)=-\sum_{\partial K_{mk}\atop m>k}\sum_{e_{jl}\subset\partial K_{mk}\atop j>l}&\{\int_{ e_{jl}}
\alpha_jw_j\nabla u_j\cdot v_l\mathbf{n}_lds+ \int_{e_{jl}}
\alpha_jw_j\nabla v_j\cdot u_l\mathbf{n}_lds\\
+&\int_{ e_{jl}}
\alpha_lw_l\nabla u_l\cdot v_j\mathbf{n}_jds+ \int_{e_{jl}}
\alpha_lw_l\nabla v_l\cdot u_j\mathbf{n}_jds\\
+&\eta\int_{ e_{jl}}\frac{W_e}{h_e}
u_jv_lds+\eta \int_{e_{jl}}\frac{W_e}{h_e}
u_lv_jds \},
\end{split}
\label{eq:interface}
\end{align}
where $u_j=u|_{\tau_j}\,\,{\rm and}\,\,u_l=u|_{\tau_l},\,\,{\rm with}\,\,e_{jl}{\rm \,\,being \,\,the\,\, common\,\, edge\,\, between\,\, } \tau_j\,\, {\rm and} \,\,\tau_l.$\\
\indent Observe that the relationship between
the bilinear form on the fine space and the bilinear forms on the local subspaces is given by
\begin{align}
a_{h}(u,v)=\sum_{i=1}^Na_{i}(u_i,v_i)+I(u,v) \,\,\, \forall u,v \in V^h, \,\,u_i, v_i\in V_i^h,
\label{eq:relation}
\end{align}
where $u=\sum_{i=1}^Nu_i,$ and $v=\sum_{i=1}^Nv_i.$\\[\baselineskip]
\indent For the proof of (\ref{eq:relation}), one only need to compare the terms in $a_h(u,v)$ with those in $\sum_{i=1}^Na_{i}(u_i,v_i),$ $\forall u,v \in V^h, \,\,u_i, v_i\in V_i^h.$  \\
\indent The next lemma which states the Poincar\'{e} and a trace inequality for our discontinuous case, has been proved in \cite{X.O.}.

\begin{lem}
 Let $D$ be a convex domain, $\{\mathcal{T}_D\}$ be a family of partitions
of $D$,  and $D=\bigcup_{\tau\in \mathcal{T}_D}\tau$ with $diam(\tau)\approx h$,
then for any $u\in V_D=\Pi_{\tau\subset \mathcal{T}_D}H^1(\tau),$ with
$\overline{u}=\frac{1}{|D|}\int_D u dx$ being the average value of $u$ over $D,$ we have
$$\|u-\overline{u}\|_{0,D}\lesssim diam(D)|u|_{1,h,D},$$
where $$|u|_{1,h,D}^2=\sum_{\tau\subset D}\|\nabla u\|_{0,\tau}^2+\sum_{e_{jl}\subset D\atop j>l}h^{-1}\|[u]\|_{0,e_{jl}}^2,$$
and $[u]$ denote the jump on the edge $e_{jl}.$
Consequently, if $diam(D)\approx O(H),$ we have
$$||u||_{0,\partial D}^2\lesssim H^{-1}||u||_{0,D}^2+H|u|_{1,h,D}^2.$$
\label{lemma:p}
\end{lem}
\begin{rem}
{\rm
Note that, in the above Lemma, the condition on the convexity of the domain $D$ is actually too strict. It is stated in \cite{X.O.} that this can be dropped.
}
\end{rem}
\indent Define the restriction operator $R_H: V^h(\Omega)\rightarrow V_0(\Omega)$ as follows:
$$ R_H u(x_p)=\overline{u}_p,\,\, \forall u\in V^h(\Omega),\,\,\forall x_p\in \mathcal{N}^H(K),\,\,\forall K\in\mathcal{T}_H,$$
where $\overline{u}_p=\frac{1}{|\omega_p|}\int_{\omega_p} u dx,$ with $\omega_p$ being the union of elements sharing $x_p,$ as shown in the figure below.

\begin{figure*}[htb]
  \centering
  \includegraphics[width=2.5cm, height=2.5cm]{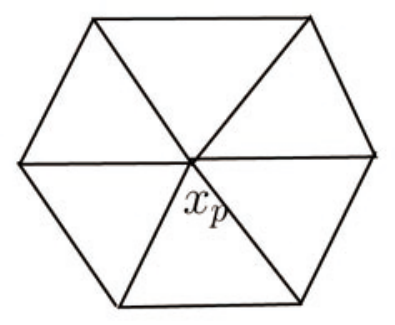}
  \centering
\end{figure*}

\begin{lem}

For the restriction operator defined above, the following approximation and stability properties hold
$$ \|u-R_H u\|_{0,K}^2\leq CH^2|u|_{1,h,\omega_K}^2, \,\,\forall u\in V^h(\Omega).$$
$$ a_{0}(R_Hu,R_Hu)\leq C \max\big{(}\gamma(\alpha),\eta\beta(\alpha)\big{)}a_{h}(u,u), \,\,\forall u\in V^h(\Omega).$$
Particularly, if $\alpha(x)\equiv1,\,\,\forall x\in\Omega,$ and $\eta=1,$ then we have
$$ a_{0}(R_Hu,R_Hu)\leq C\max\big{(}\gamma(1),\beta(1)\big{)}a_{h}(u,u), \,\,\forall u\in V^h(\Omega),$$
where $\omega_K=\bigcup_{p\in\mathcal{N}^H(K)}\omega_p,$ and the constant $C$ is independent of $h$ and $H.$
\label{lemma:approx}
\end{lem}
\begin{proof}
Since
\begin{align}
\overline{u}_p=\frac{1}{|\omega_p|}\int_{\omega_p} u dx\leq
C\frac{1}{|\omega_p|}\|u\|_{0,\omega_p}|\omega_p|^{\frac{1}{2}}\leq C\frac{1}{|\omega_p|^{\frac{1}{2}}}\|u\|_{0,\omega_p},\label{eq:up}
\end{align}
we have
\begin{equation*}
\begin{aligned}
\|u-R_H u\|_{0,K}^2
\leq &C\|u\|_{0,K}^2+C\|R_Hu\|_{0,K}^2\\
\leq &C\|u\|_{0,K}^2+C\sum_{x_p\in \mathcal{N}^H(K)}\frac{1}{|\omega_p|}\|u\|_{0,\omega_p}^2
\leq C\|u\|_{0,K}^2+C\frac{1}{|K|}\|u\|_{0,\omega_K}^2,
\end{aligned}
\end{equation*}
where in the last inequality we have used $|\omega_p|>|K|$ and $\omega_p$ is part of $\omega_K.$\\
 \indent Since our interpolation operator $R_H$ keeps constants unchanged, we can
take $u-\frac{1}{|\omega_K|}\int_{\omega_K} u dx$ instead of $u$, and the approximation
property holds due to the Poincar\'{e} inequality in Lemma \ref{lemma:p}.\\
\indent Next, we prove the stability property. Let
\begin{align*}
\begin{split}
a_{0}(R_Hu,R_Hu)=&\sum_K\int\alpha\nabla R_Hu\cdot
\nabla R_Hudx\\&-2\sum_{\partial K_{mk}\subset\Omega\atop m>k}\sum_{e_{jl}\subset\partial
K_{mk}\atop j>l}\int_{e_{jl}}\{\alpha \nabla R_Hu\}_w^e[R_Hu]ds\\ &+\eta\sum_{\partial K_{mk}\subset\Omega\atop m>k}\sum_{e_{jl}\subset \partial
K_{mk}\atop j>l}\int_{e_{jl}} \frac{W_e}{h_e} [R_Hu][R_Hu]ds\\&-2\sum_{\partial K\subset\partial\Omega}\sum_{e\subset \partial K}\int_ew_e\alpha_e \nabla (R_Hu)_e\cdot (R_Hu)_e\mathbf{n}_eds\\&+\eta\sum_{\partial K\subset\partial\Omega}\sum_{e\subset \partial K}\int_e \frac{W_e}{h_e} (R_Hu)_e(R_Hu)_eds
\\=&I_1-2I_2+I_3-2I_4+I_5.
\end{split}
\end{align*}
\indent We begin by estimating the term $I_1.$ For each $ K\in \mathcal{T}_H,$
with $\widehat{u}=u-\frac{1}{|\omega_K|}\int_{\omega_K} u dx,$  it follows from
the Poincar\'{e} inequality in Lemma \ref{lemma:p} that
\begin{align*}
\begin{split}
|R_Hu|_{1,\alpha,K}^2=|R_H\widehat{u}|_{1,\alpha,K}^2
\lesssim&\sum_{x_p\in\mathcal{N}^H(K)}(|\omega_p|^{-1}||\widehat{u}||_{0,\omega_p}^2)|\phi_{p,K}|_{1,\alpha,K}^2\\
\lesssim&\gamma(\alpha)|K|^{-1}||\widehat{u}||_{0,\omega_K}^2\lesssim\gamma(\alpha)|u|_{1,h,\omega_K}^2.
\end{split}
\end{align*}
\indent Summing over the elements $K\in\mathcal{T}_H,$ we get
\begin{align}
I_1\lesssim\gamma(\alpha)\sum_K|u|_{1,h,\omega_K}^2\lesssim \gamma(\alpha)a_{h}(u,u).
\label{eq:est1}
\end{align}
\indent Next, we estimate the term $I_2.$ Let $v=R_Hu,$ then by the Cauchy-Schwarz inequality
\begin{align*}
\begin{split}
|I_2|=&\sum_{\partial K_{mk}\subset\Omega\atop m>k}\sum_{e_{jl}\subset \partial K_{mk}\atop j>l}\int_{e_{jl}}\{\alpha \nabla v\}_w^e[v]ds\\
\lesssim&\sum_{\partial K_{mk}\subset\Omega\atop m>k}\sum_{e_{jl}\subset \partial K_{mk}\atop j>l}(\int_{e_{jl}}|\{\alpha\nabla v\}_w^e|^2ds)^{1/2}(\int_{e_{jl}}|[v]|^2ds)^{1/2}\\
\lesssim&\sum_{\partial K_{mk}\subset\Omega\atop m>k}\sum_{e_{jl}\subset
\partial K_{mk}\atop j>l}(\int_{e_{jl}}\frac{h_e}{W_e}|\{\alpha \nabla v\}_w^e|^2ds)+\sum_{\partial
K_{mk}\subset\Omega\atop m>k}\sum_{e_{jl}\subset \partial K_{mk}\atop j>l}\int_{e_{jl}}\frac{W_e}{h_e}|[v]|^2ds.
\end{split}
\end{align*}
\indent From the definition of $\{\cdot\}_w^e$, and (\ref{eq:weight-ineq}), we have
\begin{align*}
\begin{split}
|I_2|\lesssim&\sum_{\partial K_{mk}\subset\Omega\atop m>k}\sum_{e_{jl}\subset \partial K_{mk}\atop j>l}(\int_{e_{jl}}\frac{h_e}{W_e}|w_-^e\alpha_-^e \nabla v_-^e+w_+^e\alpha_+^e \nabla v_+^e|^2ds)\\
&+\sum_{\partial K_{mk}\subset\Omega\atop m>k}\sum_{e_{jl}\subset \partial K_{mk}\atop j>l}\int_{e_{jl}}\frac{W_e}{h_e}|[v]|^2ds\\
\lesssim&\sum_{\partial K_{mk}\subset\Omega\atop m>k}\sum_{e_{jl}\subset \partial K_{mk}\atop j>l}(\int_{e_{jl}}\ h_e \max(\frac{(w_-^e)^2\alpha_-^e}{W_e},\frac{(w_+^e)^2\alpha_+^e}{W_e})(\alpha_-^e |\nabla v_-^e|^2+\alpha_+^e|\nabla v_+^e|^2)ds)\\
&+\sum_{\partial K_{mk}\subset\Omega\atop m>k}\sum_{e_{jl}\subset \partial K_{mk}\atop j>l}\int_{e_{jl}}\frac{W_e}{h_e}|[v]|^2ds\\
\lesssim&\sum_{\partial K_{mk}\subset\Omega\atop m>k}\sum_{e_{jl}\subset
\partial K_{mk}\atop j>l}(\int_{e_{jl}} h_e\alpha_-^e |\nabla v_-^e|^2ds+\int_{e_{jl}} h_e\alpha_+^e
|\nabla v_+^e|^2ds)\\
&+\sum_{\partial K_{mk}\subset\Omega\atop m>k}\sum_{e_{jl}\subset \partial K_{mk}\atop j>l}\int_{e_{jl}}\frac{W_e}{h_e}|[v]|^2ds .
\end{split}
\end{align*}
\indent Let $\tau_+^e$ and $\tau_-^e$ be the two fine elements sharing the edge $e,$ and note
that $\nabla v_+^e$ and $\nabla v_-^e$ are constants on $\tau_+^e$ and $\tau_-^e$ respectively. Then
\begin{align}
\begin{split}
|I_2|\lesssim&\sum_{\partial K_{mk}\subset\Omega\atop m>k}\sum_{e_{jl}\subset
\partial K_{mk}\atop j>l}(\int_{\tau_-^e}\alpha_-^e |\nabla v_-^e|^2ds+\int_{\tau_+^e}\alpha_+^e
|\nabla v_+^e|^2ds)\\
&+\sum_{\partial K_{mk}\subset\Omega\atop m>k}\sum_{e_{jl}\subset \partial K_{mk}\atop j>l}\int_{e_{jl}}\frac{W_e}{h_e}|[v]|^2ds\\
\lesssim&\sum_{\partial K_{mk}\subset\Omega\atop m>k}\sum_{e_{jl}\subset
\partial K_{mk}\atop j>l}(\int_{\tau_+^e\cup\tau_-^e}|\alpha^{1/2} \nabla v|^2dx)+\sum_{\partial
K_{mk}\subset\Omega\atop m>k}\sum_{e_{jl}\subset \partial K_{mk}\atop j>l}\int_{e_{jl}}\frac{W_e}{h_e}|[v]|^2ds\\
\lesssim&\sum_K\int_K|\alpha^{1/2} \nabla R_Hu|^2dx+\sum_{\partial K_{mk}\subset\Omega\atop m>k}\sum_{e_{jl}\subset
\partial K_{mk}\atop j>l}\int_{e_{jl}}\frac{W_e}{h_e}|[R_Hu]|^2ds\lesssim I_1+I_3.
\end{split}
\label{eq:est2}
\end{align}

\indent We now consider the term $I_3.$ First we note that, for all$\,\,\partial K_{mk},\,\,m>k,$ we have
\begin{align*}
\begin{split}
&\eta\sum_{e_{jl}\subset\partial  K_{mk}\atop j>l}\int_e\frac{W_e}{h_e}|[R_Hu]|^2ds\\ \lesssim &\eta\sum_{e_{jl}\subset\partial  K_{mk}\atop j>l}\sum_{x_p\in\mathcal{N}^H(K_m)\atop x_p\in\overline{\partial K_{mk}}}\int_{e_{jl}}\frac{W_e}{h_e}|[\overline{u}_p\phi_p]|^2ds\lesssim\eta\beta(\alpha)\sum_{x_p\in \mathcal{N}^H(K_m)\atop x_p\in\overline{\partial K_{mk}}}|\overline{u}_p|^2\\
\lesssim& \eta\beta^I(\alpha)\sum_{x_p\in \mathcal{N}^H(K_m)\atop x_p\in\overline{\partial K_{mk}}}\frac{1}{|\omega_p|}||u||_{0,\omega_p}^2\lesssim \eta\beta^I(\alpha)\frac{1}{|K_m|}||u||_{0,\omega_{K_m}}^2.
\end{split}
\end{align*}
\indent Taking $\widehat{u}=u-\frac{1}{|\omega_{K_m}|}\int_{\omega_{K_m}} u dx$ instead of $u$ in the above equation, it follows from the Poincar\'{e} inequality in Lemma \ref{lemma:p} that
\begin{align*}
\eta\sum_{e_{jl}\subset\partial K_{mk}\atop j>l}\int_{e_{jl}}\frac{W_e}{h_e}|[R_Hu]|^2ds=&\eta\sum_{e_{jl}\subset\partial  K_{mk}\atop j>l}\int_{e_{jl}}\frac{W_e}{h_e}|[R_H\widehat{u}]|^2ds\\ \lesssim &\eta\beta^I(\alpha)\frac{1}{|K_m|}||\widehat{u}||_{0,\omega_{K_m}}^2\lesssim\eta\beta^I(\alpha)|u|_{1,h,\omega_{K_m}}^2.
\end{align*}
\indent Hence, we have
\begin{align}
\begin{split}
I_3=\eta\sum_{\partial K_{mk}\atop m>k}\sum_{e_{jl}\subset\partial K_{mk}}\int_{e_{jl}}\frac{W_e}{h_e}|[R_Hu]|^2ds\lesssim\eta\beta^I(\alpha)\sum_m|u|_{1,h,\omega_{K_m}}^2\lesssim\eta\beta^I(\alpha)a_{h}(u,u).
\end{split}
\label{eq:est3}
\end{align}
\indent Finally, for the terms on the boundary $\partial \Omega,$ we apply the same techniques, and we get
 \begin{align}
 |I_4|\lesssim I_1+I_5,
 \label{eq:est4}
 \end{align}
 and
\begin{align}
\begin{split}
I_5=\eta\sum_{\partial K\subset\partial\Omega}\sum_{e\subset \partial K}\int_e \frac{W_e}{h_e} (R_Hu)_e(R_Hu)_eds
\lesssim&\eta\beta^B(\alpha)\sum_m|u|_{1,h,\omega_{K_m}}^2\\ \lesssim&\eta\beta^B(\alpha)a_{h}(u,u).
\end{split}
\label{eq:est5}
\end{align}
\indent The stability estimate for the general case thus follows from (\ref{eq:est1}), (\ref{eq:est2}), (\ref{eq:est3}), (\ref{eq:est4}) and (\ref{eq:est5}) above. The estimate for the particular case holds naturally.
\end{proof}
\indent  We can now give an explicit bound for the condition number of our two level nonoverlapping additive Schwarz method.
\begin{thm}
For all $u\in V^h$, there exists $u_0\in V_0$ and $u_i\in V_i^h$, $1\leq i\leq N,$ such that
$$\sum_{i=0}^Na_{i}(u_i,u_i)\leq C\lambda a_{h}(u,u).$$
Consequently, we have $\kappa(T)\lesssim \lambda,$ where $\kappa(T)$ denotes the condition number of the additive Schwarz operator $T$ as defined in (\ref{eq:ddm}), and $\lambda$ is given as
$$\lambda=\eta\max_{\partial K_{mk}\atop m>k} \max_{e\subset \partial K_{mk}}W_e\gamma(1)\frac{H}{h}+\max\big{(}\gamma(\alpha), \eta\beta(\alpha)\big{)},$$ where $C$ is a constant independent of $h,$ $H,$ and $\alpha.$
\label{thm:nonoverlapping}
\end{thm}
\begin{proof}
We need to verify the three assumptions of the Schwarz framework \cite{S.B.G.,T.W.}.
More precisely, we need to estimate the three constants $C_0^2$, $\rho(M)$ and $\omega$ which corresponds to the three assumptions. The bound
of the condition number is then given as:
\[
\kappa (T) \leq C_0^2 \omega (\rho(M) +1). \\[\baselineskip]
\]
{\bf The first assumption} of the Schwarz framework, asks for an estimate of the smallest $C_0^2$ such that
\[
\sum_{i=0}^Na_{i}(u_i,u_i)\lesssim C_0^2a_{h}(u,u).
\]

\indent For all $u\in V^h,$ let $u_0=R_H u,$ $z=u-u_0$ and $u_i=z|_{\Omega_i}.$ From (\ref{eq:relation}), we know that
$$a_{h}(u-u_0,u-u_0)=\sum_{i=1}^Na_{i}(u_i,u_i)+I(u-u_0,u-u_0).$$
Hence,
\begin{align}
\begin{split}
|\sum_{i=1}^Na_{i}(u_i,u_i)|&\leq|a_{h}(u-u_0,u-u_0)-I(u-u_0,u-u_0)|\\
&\leq |a_{h}(u-u_0,u-u_0)|+|I(u-u_0,u-u_0)|=I_1+I_2.\label{eq:te}
\end{split}
\end{align}
\indent Using the fact that $a_{h}(u_0,u_0)=a_{0}(u_0,u_0)$ and Lemma \ref{lemma:approx}, it then follows
\begin{align}
I_1=|a_{h}(u-u_0,u-u_0)|\lesssim a_{h}(u,u)+a_{0}(u_0,u_0)\lesssim \max\big{(}\gamma(\alpha),\eta\beta(\alpha)\big{)}a_{h}(u,u).\label{eq:estI1}
\end{align}
\indent Next, we estimate the term $I_2.$
\begin{align}
\begin{split}
I_2=|I(z,z)|\lesssim&\sum_{\partial K_{mk}\atop m>k}\sum_{e_{jl}\subset\partial K_{mk}\atop j>l}|\int_{ e_{jl}}
\alpha_jw_j\nabla z_j\cdot z_l\mathbf{n}_lds|\\
&+\sum_{\partial K_{mk}\atop m>k}\sum_{e_{jl}\subset\partial K_{mk}\atop j>l}|\int_{e_{jl}}
\alpha_lw_l\nabla z_l\cdot z_j\mathbf{n}_jds|\\
&+\eta\sum_{\partial K_{mk}\atop m>k}\sum_{e_{jl}\subset\partial K_{mk}\atop j>l}|\int_{ e_{jl}}\frac{W_e}{h_e}z_jz_lds |=I_{21}+I_{22}+I_{23}.\label{eq:estI2}
\end{split}
\end{align}
\indent To estimate the terms $I_{21}$ and $I_{22},$ we use similar techniques as in Lemma \ref{lemma:approx}, which gives
\begin{align}
\begin{split}
I_{21}
\lesssim &\sum_K||\alpha^{1/2}\nabla z||_{0,K}^2+\sum_{\partial K_{mk}\atop m>k}\sum_{e_{jl}\subset\partial K_{mk}\atop j>l}\frac{W_e}{h_e}(\int_{e_{jl}}|z_l|^2ds)=I_{211}+I_{212},\label{eq:estI21}
\end{split}
\end{align}
and
\begin{align}
I_{22}\lesssim &\sum_K||\alpha^{1/2}\nabla z||_{0,K}^2+\sum_{\partial K_{mk}\atop m>k}\sum_{e_{jl}\subset\partial K_{mk}\atop j>l}\frac{W_e}{h_e}(\int_{e_{jl}}|z_j|^2ds)=I_{221}+I_{222}. \label{eq:estI22}
\end{align}
\indent Using the Cauchy-Schwarz inequality for the term $I_{23},$ we have
\begin{align}
\begin{split}
I_{23}=&\eta\sum_{\partial K_{mk}\atop m>k}\sum_{e_{jl}\subset\partial K_{mk}\atop j>l}|\int_{ e_{jl}}\frac{W_e}{h_e}z_jz_lds \}| \\
\lesssim&\eta\sum_{\partial K_{mk}\atop m>k}\sum_{e_{jl}\subset\partial K_{mk}\atop j>l}(\int_{ e_{jl}}\frac{W_e}{h_e}|z_j|^2ds)^{1/2}(\int_{ e_{jl}}\frac{W_e}{h_e}|z_l|^2ds)^{1/2}\\
\lesssim&\eta \sum_{\partial K_{mk}\atop m>k}\sum_{e_{jl}\subset\partial K_{mk}\atop j>l}\int_{ e_{jl}}\frac{W_e}{h_e}|z_j|^2ds+\eta\sum_{\partial K_{mk}\atop m>k}\sum_{e_{jl}\subset\partial K_{mk}\atop j>l}\int_{ e_{jl}}\frac{W_e}{h_e}|z_l|^2ds\\
=&I_{231}+I_{232}. \label{eq:estI23}
\end{split}
\end{align}
\indent So, we need to estimate the six terms $I_{211},$ $I_{212},$ $I_{221},$ $I_{222},$ $I_{231},$ $I_{232}.$ We begin by estimating the terms $I_{211}$ and $I_{221}.$ We note that
\begin{align}
\begin{split}
|u-u_0|_{1,\alpha,K}^2
\lesssim&|u|_{1,\alpha,K}+\sum_{x_p\in\mathcal{N}^H(K)}(|\omega_p|^{-1}||u||_{0,\omega_p}^2)|\phi_{p,K}|_{1,\alpha,K}^2\\
\lesssim&|u|_{1,\alpha,K}+\gamma(\alpha)\frac{1}{|K|}||u||_{0,\omega_K}^2.  \label{eq:westu-u0}
\end{split}
\end{align}
\indent Let $\widehat{u}=u-\frac{1}{|\omega_K|}\int_{\omega_K} u dx$ and $\widehat{u}_0=R_H\widehat{u}.$ Also note that $u-u_0=\widehat{u}-\widehat{u}_0$ since $R_H$ preserves constants. It follows from (\ref{eq:westu-u0}) and Lemma \ref{lemma:p} that
\begin{align*}
|u-u_0|_{1,\alpha,K}^2=|\widehat{u}-\widehat{u}_0|_{1,\alpha,K}^2\lesssim|u|_{1,\alpha,K}^2+\gamma(\alpha)\frac{1}{|K|}||\widehat{u}||_{0,\omega_K}^2\lesssim\gamma(\alpha)|u|_{1,h,\alpha,\omega_K}^2.
\end{align*}
\indent We then have
\begin{align*}
I_{211}=I_{221}=\sum_K|u-u_0|_{1,\alpha,K}^2\lesssim\gamma(\alpha)|u|_{1,h,\alpha,\omega_K}^2\lesssim\gamma(\alpha)a_h(u,u).
\end{align*}
\indent Next, we estimate the term $I_{232}.$ Estimates for the terms $I_{212},$ $I_{222}$ and $I_{231}$ are similar.
\begin{align}
\begin{split}
I_{232}\lesssim &\eta\sum_{\partial K_{mk}\atop m>k}\max_{e_{jl}\subset\partial K_{mk}\atop j>l} \frac{W_e}{h_e}\sum_{e_{jl}\subset\partial K_{mk}\atop j>l}\int_{ e_{jl}}|z_l|^2ds\\
\lesssim &\eta\sum_{\partial K_{mk}\atop m>k}\max_{e_{jl}\subset\partial K_{mk}\atop j>l} \frac{W_e}{h_e}(||z_m||_{0,\partial K_{m}}^2+||z_k||_{0,\partial K_{k}}^2).\label{eq:I232}
\end{split}
\end{align}
\indent Using the trace inequality in Lemma \ref{lemma:p}, we have
\begin{align}
\begin{split}
I_{232}\lesssim &\eta\sum_{\partial K_{mk}\atop m>k}\max_{e_{jl}\subset\partial K_{mk}\atop j>l} \frac{W_e}{h_e}(H^{-1}||z_m||_{0,K_m}^2+H|z_m|_{1,h,K_m}^2)\\
&+\eta\sum_{\partial K_{mk}\atop m>k}\max_{e_{jl}\subset\partial K_{mk}\atop j>l} \frac{W_e}{h_e}(H^{-1}||z_k||_{0,K_k}^2+H|z_k|_{1,h,K_k}^2)\\
\lesssim &\eta\sum_{\partial K_{mk}\atop m>k}\max_{e_{jl}\subset\partial K_{mk}\atop j>l} \frac{W_e}{h_e}(H^{-1}||u-u_0||_{0,K_m}^2+H|u-u_0|_{1,h,K_m}^2)\\
&+\eta\sum_{\partial K_{mk}\atop m>k}\max_{e_{jl}\subset\partial K_{mk}\atop j>l} \frac{W_e}{h_e}(H^{-1}||u-u_0||_{0,K_k}^2+H|u-u_0|_{1,h,K_k}^2).\\
\label{eq:estI232}
\end{split}
\end{align}
\indent We recall that $u_0$ is a continuous function on each $K\in\mathcal{T}_H,$  which implies that
\begin{align}
\begin{split}
|u-u_0|_{1,h,K}^2\lesssim &|u|_{1,h,K}^2+|u_0|_{1,h,K}^2\lesssim|u|_{1,h,K}^2+|u_0|_{1,K}^2\\
\lesssim&|u|_{1,h,K}+\sum_{x_p\in\mathcal{N}^H(K)}(|\omega_p|^{-1}||u||_{0,\omega_p}^2)|\phi_{p,K}|_{1,K}^2\\
\lesssim&|u|_{1,h,K}+\gamma(1)\frac{1}{|K|}||u||_{0,\omega_K}^2.  \label{eq:estu-u0}
\end{split}
\end{align}
\indent Again let $\widehat{u}=u-\frac{1}{|\omega_K|}\int_{\omega_K} u dx$ and $\widehat{u}_0=R_H\widehat{u},$ and note that $u-u_0=\widehat{u}-\widehat{u}_0$ ($R_H$ preserves constants),  we have from (\ref{eq:estu-u0}) that
\begin{align}
|u-u_0|_{1,h,K}^2=|\widehat{u}-\widehat{u}_0|_{1,h,K}^2\lesssim|u|_{1,h,K}^2+\gamma(1)\frac{1}{|K|}||\widehat{u}||_{0,\omega_K}^2\lesssim\gamma(1)|u|_{1,h,\omega_K}.
\label{eq:festu-u0}
\end{align}
\indent Using the approximation property of $R_H$ in Lemma \ref{lemma:approx}, as well as (\ref{eq:festu-u0}), it follows from (\ref{eq:estI232}) that
\begin{align*}
\begin{split}
I_{232}\lesssim &\eta \sum_{\partial K_{mk}\atop m>k}\max_{e_{jl}\subset\partial K_{mk}\atop j>l}
\frac{W_e}{h_e}(H|u|_{1,h,\omega_{K_m}}^2+H\gamma(1)|u|_{1,h,\omega_{K_m}}^2)\\
&+\eta \sum_{\partial K_{mk}\atop m>k}\max_{e_{jl}\subset\partial K_{mk}\atop j>l}
\frac{W_e}{h_e}(H|u|_{1,h,\omega_{K_k}}^2+H\gamma(1)|u|_{1,h,\omega_{K_k}}^2)\\
\lesssim &\eta\max_{\partial K_{mk}\atop m>k} \max_{e_{jl}\subset \partial K_{mk}\atop j>l}W_e\gamma(1)\frac{H}{h}a_{h}(u,u).
\end{split}
\end{align*}
\indent Having all the estimates of the terms $I_{21},$ $I_{22}$ and $I_{23}$ together in (\ref{eq:estI2}), we get
\begin{align*}
I_2\lesssim\eta\max_{\partial K_{mk}\atop m>k} \max_{e_{jl}\subset \partial K_{mk}\atop j>l}W_e\gamma(1)\frac{H}{h}a_{h}(u,u)+\gamma(\alpha)a_{h}(u,u).
\end{align*}
\indent Combing (\ref{eq:te}), the estimates for $I_1$ and $I_2,$ and Lemma \ref{lemma:approx} we have
$$\sum_{i=0}^Na_{i}(u_i,u_i)\lesssim \Big{(} \eta\max_{\partial K_{mk}\atop m>k}
\max_{e_{jl}\subset \partial K_{mk}\atop j>l}W_e\gamma(1)\frac{H}{h}+\max\big{(}\gamma(\alpha),\eta\beta(\alpha)\big{)}\Big{)}a_{h}(u,u).$$ \\
{\bf The second assumption} of the Schwarz framework, requires a bound for the spectral
radius $\rho(M)$ of the $N\times N$ matrix $M,$ whose elements $M_{ij}$ are defined in terms of a strengthen Cauchy Schwarz inequality:
let $0\leq M_{ij}\leq 1$ be the minimum values such that
\[
|a_{h}(u_i,u_j)|\leq M_{ij}a_{h}(u_i,u_i)^{\frac{1}{2}}a_{h}(u_j,u_j)^{\frac{1}{2}},
\]
where $u_i\in V_i^h,\,\,u_j\in V_j^h,\,\,1\leq i,j\leq N.$ \\[\baselineskip]
\indent Note that in our definitions above
$$|a_{h}(u_i,u_j)|=0, \,\,{\rm if}\,\, {\rm meas}(\partial \Omega_{ij})=0,\,\,1\leq i,j\leq N,$$
where $\partial \Omega_{ij}=\partial \Omega_i\cap\partial \Omega_j.$ This is because, according to (\ref{eq:bf}), all the terms become zero since functions $u_i$ and $u_j$ have no common support.
For the remaining case we can take $M_{ij}= 1.$ It follows at once from Gershgorin's circle
theorem that
$$\rho(M)\leq N_c+1,$$ where $N_c$ is the maximum number of subdomains adjacent to any subdomain.\\[\baselineskip]
{\bf The third assumption} asks for $\omega$ such that
$$a_{h}(u_i,u_i)\leq\omega a_{i}(u_i,u_i),\,\,\forall u_i\in V_i^h,,\,\,0\leq i\leq N.$$
Since we use exact bilinear form for the subproblems, $\omega=1.$\\
\indent Our theorem is proved since the analysis of the three assumptions is complete.
\end{proof}

\subsubsection{\rm Overlapping additive Schwarz method}
In this section, we analyze the overlapping version of our additive Schwarz method.
As will be shown in Theorem \ref{thm:overlapping}, the condition number bound is not only dependent on
the indicators $\gamma(\alpha)$ and $\beta(\alpha)$ but also on the partition robustness
indicator $\pi(\alpha)$ borrowed from \cite{I.P.R.}, which describes the relationship between the subdomain overlap
and the coefficients $\alpha.$
We can control this indicator by choosing a suitable overlap.\\
\indent Given an overlapping partition $\{\Omega_i'\}$ with $\delta_i>0,$ $i=1, \cdot\cdot\cdot, N,$ let $\{\chi_i\}_{i=1}^N$ be
the partition of unity subordinate to $\{\Omega_i'\},$ c.f., \cite{D.N.}, for which the following
property holds,
\begin{align}
\chi_i(x)=1\,\,{\rm and\,\,} \nabla \chi_i=0,\,\, {\rm for\,\, all\,\,} x\in \Omega_i'^o=\{x\in \Omega_i': x\not\in\overline{\Omega_j'}\,\,{\rm for\,\,any\,\,} j\neq i \}.\label{eq:pu}
\end{align}

\indent $\mathbf{Partition\,\, robustness\,\, indicator}$\,\,\cite{I.P.R.}. For a particular partition of unity $\{\chi_i \}$ subordinate to the covering $\{\Omega_i'\},$ let
$$\pi(\alpha, \{\chi_i\})=\max_{i=1}^N \{\delta_i^2\| \alpha |\nabla \chi_i|^2\|_{L_\infty(\Omega)} \},$$
then the partition robustness indicator is defined as
$$\pi(\alpha)=\inf_{\{\chi_i\}\in\Pi(\{\Omega_i'\})}\pi(\alpha,\{\chi_i\}),$$
where $\Pi(\{\Omega_i'\})$ denote the set of all the partitions of
unity $\{\chi_i\}$ subordinate to the cover $\{\Omega_i'\}$.\\[\baselineskip]

\indent Following similarly as the nonoverlapping case, we use $V_0$ as the coarse space with bilinear form $a_0(\cdot,\cdot)$, the subspaces $V_i^h$ and the local bilinear forms $a_{i}(\cdot,\cdot),\,\,1\leq i\leq N,$ can be got by taking $B_i=\Omega_i'$ in (\ref{eq:ls}) and (\ref{eq:lbf}) respectively.\\
\indent The next lemma gives us an estimate on the boundary layer.

\begin{lem}
Let $D$ be a convex domain with ${\rm diam}(D)= O(H),$ $\mathcal{T}_{D}$ be a
family of partitions over $D$, and $D=\bigcup_{\tau\in \mathcal{T}_{D}}\overline{\tau}$
with ${\rm diam}(\tau)= O(h).$ Let $0<\mu\leq H,$ then for any $u\in V_{D}=\Pi_{\tau\subset
\mathcal{T}_{D}}H^1(\tau),$ then we have
$$||u||_{0,{D}_\mu}^2\leq C[\mu H^{-1}||u||_{0,{D}}^2+\mu(\mu+H)|u|_{1,h,{D}}^2],$$
where ${D}_\mu=\{x\in {D}:{\rm dist}(x,\partial {D})\leq \mu\}$ denotes the
boundary layer of ${D}$ with width $\mu,$ and $|u|_{1,h,{D}}^2=\sum_{\tau\subset {D}}
||\nabla u||_{0,\tau}^2+\sum_{e_{jl}\subset {D}\atop j>l}h_{jl}^{-1}\int_{e_{jl}}[u]^2ds.$
\label{lemma:ble}
\end{lem}
The proof of this lemma can be found in \cite{X.O.}, again the assumption on the convexity of the domain $D$ can be dropped, c.f., \cite{X.O.}.\\

\begin{thm}
For all $u\in V^h$, there exists $u_0\in V_0$ and $u_i\in V_i^h$, $1\leq i\leq N,$ such that
$$\sum_{i=0}^Na_{i}(u_i,u_i)\leq C\lambda a_{h}(u,u),$$
Consequently, we have $\kappa(T)\lesssim \lambda,$ where $\kappa(T)$ denotes the condition number of the additive Schwarz operator $T$ as defined in (\ref{eq:ddm}), and $\lambda$ is given as
$$\lambda=\pi(\alpha)\max\big{(}\gamma(1),\beta(1)\big{)}\max_i\frac{H_i}{\delta_i}+\max\big{(}\gamma(\alpha), \eta\beta(\alpha)\big{)},$$ and $C$ is a constant independent of $h,$ $H,$ and $\alpha.$\label{thm:overlapping}
\end{thm}
\begin{proof}
Again, we use the Schwarz framework to prove this theorem. Like in the nonoverlapping case, we estimate the three parameters $C_0^2$, $\rho(M)$ and $\omega$ .\\[\baselineskip]
\indent We first estimate $C_0^2.$ For all $u\in V^h$, we may choose $u_0=R_Hu$ and $u_i=I_h(\chi_i(u-u_0)),$ where $I_h$ is the usual Lagrange interpolation operator. Let $w=u-u_0$ and $\{\chi_i\}$ be the partition of unity subordinate to the covering $\{\Omega_i'\},$ then $u_i=I_h(\chi_iw).$\\
\indent For $1\leq i\leq N,$ since $\chi_i=0$ on $\Gamma_i=\partial \Omega_i'\backslash\partial\Omega,$ we know from (\ref{eq:lbf}) that
\begin{align*}
\begin{split}
a_i(u_i,u_i)=&\sum_{\tau\subset \Omega_i'}\int_{\tau}\alpha\nabla u_i\cdot\nabla u_idx-2\sum_{ e_{jl}\subset \Omega_i'\atop j>l}\int_{e_{jl}}\{\alpha \nabla u_i\}_w^e[u_i]ds\\
&+\eta\sum_{e_{jl}\subset \Omega_i'\atop j>l}\int_{e_{jl}} \frac{W_e}{h_e}[u_i][u_i]ds-2\sum_{ e\subset \Xi_i}\int_{e}w_e\alpha_e (\nabla u_i)_e\cdot(u_i)_e\mathbf{n}_eds\\
&+\eta\sum_{e\subset \Xi_i}\int_{e} \frac{W_e}{h_e}(u_i)_e(u_i)_eds.
\end{split}
\end{align*}
\indent Using the same techniques as in Lemma \ref{lemma:approx}, we can write
\begin{align*}
\sum_{ e_{jl}\subset \Omega_i'\atop j>l}\int_{e_{jl}}\{\alpha \nabla u_i\}_w^e[u_i]ds
\lesssim&\sum_{\tau\subset\Omega_i'}\int_{\tau}\alpha\nabla u_i\cdot\nabla u_i dx+\eta\sum_{e_{jl}\subset\Omega_i'\atop j>l}\int_{e_{jl}}\frac{W_e}{h_e}[u_i]^2ds,
\end{align*}
and
\begin{align*}
\sum_{ e\subset \Xi_i}\int_{e}w_e\alpha_e (\nabla u_i)_e\cdot(u_i)_e\mathbf{n}_eds
\lesssim&\sum_{\tau\subset\Omega_i'}\int_{\tau}\alpha\nabla u_i\cdot\nabla u_i dx+\eta\sum_{e\subset \Xi_i}\int_{e}\frac{W_e}{h_e}(u_i)_e^2ds.
\end{align*}
\indent Consequently,
\begin{align*}
\begin{split}
a_i(u_i,u_i)\lesssim&\sum_{\tau\subset\Omega_i'}\int_{\tau}\alpha\nabla u_i\cdot\nabla u_i dx+\eta\sum_{e_{jl}\subset\Omega_i'\atop j>l}\int_{e_{jl}}\frac{W_e}{h_e}[u_i]^2ds+\eta\sum_{e\subset \Xi_i}\int_{e}\frac{W_e}{h_e}(u_i)_e^2ds\\ =&I_1+I_2+I_3.
\end{split}
\end{align*}
\indent We begin by estimating term $I_1.$ Since $u_i=I_h(\chi_iw),$ we have
\begin{align*}
I_1=\sum_{\tau\subset\Omega_i'}\int_{\tau}\alpha\nabla u_i\cdot\nabla u_i dx=\sum_{\tau\subset\Omega_i'}\int_{\tau}\alpha|\nabla I_h(\chi_i w)|^2 dx.
\end{align*}
\indent Let $\chi_{i,\tau}=\frac{1}{|\tau|}\int_{\tau}\chi_idx,$ then it follows from the Bramble-Hilbert Lemma, c.f., \cite{D.N.}, that,
\[
||\chi_i-\chi_{i,\tau}||_{0,\infty,\tau}\lesssim
\begin{cases}h_{\tau}|\nabla \chi_i|,\quad \ \ & \forall \tau\subset\Omega_{i,\delta_i}',\\
0,\quad \ \ & \tau\subset \Omega_i'^o,\\
\end{cases}
\]
where $\Omega_{i,\delta_i}'=\{x\in\Omega_i': dist(x,\partial \Omega_i')\leq\delta_i\}$ denotes the boundary layer of $\Omega_i'$ with width $\delta_i,$ and $\Omega_i'^o=\Omega_i'\backslash \Omega_{i,\delta_i}'$ is the interior part of $\Omega_i'.$\\
\indent Since the interpolation operator $I_h$ is stable with respect to the norm $||\cdot||_{0,\infty,\Omega},$
using the inverse inequality, we have
\begin{align*}
\begin{split}
\int_{\tau}\alpha|\nabla I_h(\chi_i w)|^2dx=&||\alpha^{1/2}\nabla(I_h(\chi_i w))||_{0,\tau}^2\\
\lesssim&||\alpha^{1/2}\nabla(I_h(\chi_{i,\tau} w))||_{0,\tau}^2+||\alpha^{1/2}\nabla I_h((\chi_i-\chi_{i,\tau}) w)||_{0,\tau}^2\\
\lesssim&||\alpha^{1/2}\nabla w||_{0,\tau}^2+h_{\tau}^{-2}||\alpha^{1/2}I_h((\chi_i-\chi_{i,\tau}) w)||_{0,\tau}^2\\
\lesssim&||\alpha^{1/2}\nabla w||_{0,\tau}^2+||\alpha^{1/2}I_h((\chi_i-\chi_{i,\tau}) w)||_{0,\infty,\tau}^2\\
\lesssim&||\alpha^{1/2}\nabla w||_{0,\tau}^2+h_{\tau}^{-2}\alpha||\chi_i-\chi_{i,\tau} ||_{0,\infty,\tau}^2||w||_{0,\tau}^2,
\end{split}
\end{align*}
\indent which implies that
\[
\int_{\tau}\alpha|\nabla I_h(\chi_i w)|^2dx\lesssim||\alpha^{1/2}\nabla w||_{0,\tau}^2+
\begin{cases}\alpha|\nabla \chi_i|^2||w||_{0,\tau}^2,\quad \ \ & \forall \tau\subset\Omega_{i,\delta_i}',\\
0,\quad \ \ & \tau\subset \Omega_i'^o.\\
\end{cases}
\]
\indent Adding this estimate across all the fine elements $\tau\subset \Omega_i'$, we have
\begin{align*}
\begin{split}
I_1=\sum_{\tau\subset\Omega_i'}\int_{\tau}\alpha\nabla I_h(\chi_i w)\cdot\nabla I_h(\chi_i w)dx
\lesssim||\alpha|\nabla \chi_i|^2||_{0,\infty,\Omega}||w||_{0,\Omega_{i,\delta_i}'}^2+|w|_{1,\alpha,\Omega_i'}^2.
\end{split}
\end{align*}
\indent By the definition of $\pi(\alpha),$ and using the estimate in Lemma \ref{lemma:ble} with $\mu=\delta_i$ and $D=\Omega_i',$ we have
\begin{align}
\begin{split}
I_1&\lesssim \pi(\alpha)\frac{1}{\delta_i^2}||w||_{0,\Omega_{i,\delta_i}'}^2+|w|_{1,\alpha,\Omega_i'}^2\\
&\lesssim \pi(\alpha)\frac{1}{\delta_i^2}(\delta_i H_i^{-1}||w||_{0,\Omega_{i}'}^2+\delta_i H_i|w|_{1,h,\Omega_i'}^2)+|w|_{1,\alpha,\Omega_i'}^2\\
&\lesssim \pi(\alpha)\delta_i^{-1} H_i^{-1}\sum_{K\cap\Omega_i'\neq\emptyset}||w||_{0,K}^2+\pi(\alpha)\frac{H_i}{\delta_i}|w|_{1,h,\Omega_i'}^2+|w|_{1,\alpha,\Omega_i'}^2\\
&\lesssim \pi(\alpha)\frac{H_i}{\delta_i}\sum_{K\cap\Omega_i'\neq\emptyset}|u|_{1,h,\omega_K}^2+\pi(\alpha)
\frac{H_i}{\delta_i}|w|_{1,h,\Omega_i'}^2+|w|_{1,\alpha,\Omega_i'}^2,\\ \label{eq:estI1_over}
\end{split}
\end{align}
where in the last inequality we have used the approximation property of $R_H$ from Lemma \ref{lemma:approx}. \\
\indent Next, we estimate the term $I_2.$ We first note that
\begin{align}
\begin{split}
|[u_i]|_{0,e_{jl}}^2&=|u_i^j-u_i^l|_{0,e_{jl}}^2=|I_h(\chi_iw^j)-I_h(\chi_iw^l)|_{0,e_{jl}}^2\\
&=|I_h(\chi_i(w^j-w^l))|_{0,e_{jl}}^2\leq|w_j-w_l|_{0,e_{jl}}^2=|[w]|_{0,e_{jl}}^2,\label{eq:estjump}
\end{split}
\end{align}
where in the above equality we have used $|\chi_i|\leq 1,\,\,1\leq i\leq N,$ and the stability property of $I_h$ w.r.t. the norm $||\cdot||_{0,\infty,\Omega}.$\\
\indent It follows from (\ref{eq:estjump}) that
\begin{align*}
\begin{split}
I_2=\eta\sum_{e_{jl}\subset\Omega_i'\atop j>l}\int_{e_{jl}}\frac{W_e}{h_e}[u_i]^2ds\lesssim\eta\sum_{e_{jl}\subset\Omega_i'\atop j>l}\int_{e_{jl}}\frac{W_e}{h_e}[w]^2ds.
\end{split}
\end{align*}
\indent Finally, for the term $I_3,$ since, by definition, $u_i=I_h(\chi_i w)$ and $\chi_i=1$ on $\Xi_i,$ we have
\begin{align*}
\begin{split}
I_3=\eta\sum_{e\subset \Xi_i}\int_{e}\frac{W_e}{h_e}(u_i)_e^2ds\lesssim\eta\sum_{e\subset \Xi_i}\int_{e}\frac{W_e}{h_e}w_e^2ds.
\end{split}
\end{align*}
\indent Combining the estimates of $I_1,$ $I_2$ and $I_3,$ $1\leq i\leq N,$ we have
\begin{align}
\begin{split}
\sum_{i=1}^Na_{i}(u_i,u_i)\lesssim &\sum_{i=1}^N\sum_{\tau\subset\Omega_i'}\int_{\tau}\alpha\nabla u_i\cdot\nabla u_i dx+\eta\sum_{i=1}^N\sum_{e_{jl}\subset\Omega_i'\atop j>l}\int_{e_{jl}}\frac{W_e}{h_e}[u_i]^2ds\\
&+\eta\sum_{i=1}^N\sum_{e\subset \Xi_i}\int_{e}\frac{W_e}{h_e}(u_i)_e^2ds\\
\lesssim& \sum_{i=1}^N \pi(\alpha)\frac{H_i}{\delta_i}\sum_{K\cap\Omega_i'\neq\emptyset}|u|_{1,h,\omega_K}^2+ \sum_{i=1}^N\pi(\alpha)\frac{H_i}{\delta_i}|w|_{1,h,\Omega_i'}^2+ \sum_{i=1}^N|w|_{1,\alpha,\Omega_i'}^2\\
&+ \eta\sum_{i=1}^N\sum_{e_{jl}\subset\Omega_i'\atop j>l}\int_{e_{jl}}\frac{W_e}{h_e}[w]^2ds
+\eta\sum_{i=1}^N\sum_{e\subset \Xi_i}\int_{e}\frac{W_e}{h_e}w_e^2ds\\
\lesssim&  \pi(\alpha)\max_i\frac{H_i}{\delta_i}a_{h}(u,u)+ \sum_{i=1}^N\pi(\alpha)\frac{H_i}{\delta_i}|u-R_Hu|_{1,h,\Omega_i'}^2+ a_{h}(w,w).\label{eq:estonlbf}
\end{split}
\end{align}
\indent Note that,
\begin{align}
\begin{split}
\sum_{i=1}^N\pi(\alpha)\frac{H_i}{\delta_i}|u-R_Hu|_{1,h,\Omega_i'}^2
\lesssim&\sum_{i=1}^N\pi(\alpha)\frac{H_i}{\delta_i}|u|_{1,h,\Omega_i'}^2+\sum_{i=1}^N\pi(\alpha)\frac{H_i}{\delta_i}|R_Hu|_{1,h,\Omega_i'}^2.
\label{eq:u-R_Hu}
\end{split}
\end{align}
\indent It follows from Lemma \ref{lemma:approx} that
\begin{align*}
\sum_{i=1}^N|R_Hu|_{1,h,\Omega_i'}^2\lesssim\max\big{(}\gamma(1),\beta(1)\big{)}a_{h}(u,u),
\end{align*}
\indent which, together with (\ref{eq:u-R_Hu}), implies that
\begin{align*}
\sum_{i=1}^N\pi(\alpha)\frac{H_i}{\delta_i}|u-R_Hu|_{1,h,\Omega_i'}^2
\lesssim\pi(\alpha)\max_i\frac{H_i}{\delta_i}\max\big{(}\gamma(1),\beta(1)\big{)}a_{h}(u,u).
\end{align*}
\indent Since $u_0=R_Hu,$ using the stability property of $R_Hu$ in Lemma \ref{lemma:approx}, we have
\begin{align*}
\begin{split}
a_{h}(w,w)=&a_{h}(u-u_0,u-u_0)\lesssim a_{h}(u,u)+a_{h}(u_0,u_0)\\
\lesssim &\max\big{(}\gamma(\alpha),\eta\beta(\alpha)\big{)}a_{h}(u,u).
\end{split}
\end{align*}
\indent Thus, by (\ref{eq:estonlbf}), we have
\begin{align*}
\begin{split}
\sum_{i=1}^Na_{i}(u_i,u_i)\lesssim\pi(\alpha)\max_i\frac{H_i}{\delta_i}\max\big{(}\gamma(1),\beta(1)\big{)}+\max\big{(}\gamma(\alpha),\eta\beta(\alpha)\big{)}a_{h}(u,u).
\end{split}
\end{align*}
\indent Note that $a_{0}(R_Hu,R_Hu)\lesssim \max\big{(}\gamma(\alpha),\eta\beta(\alpha)\big{)}a_{h}(u,u),$ which
implies that
$$C_0^2 \lesssim \pi(\alpha)\max\big{(}\gamma(1),\beta(1)\big{)}\max_i\frac{H_i}{\delta_i}+\max\big{(}\gamma(\alpha), \eta\beta(\alpha)\big{)}.$$
 \indent The other two parameters $\rho(M)$ and $\omega$ are estimated in the same way as before.
\end{proof}

\subsection{\rm Oscillatory boundary conditions}
In this section, we follow the notations used in \cite{I.P.R.}, and give an explicit bound for the indicator $\gamma(\alpha)$ with a slightly different proof. We show that if the high-conductivity region crosses the boundaries of coarse grid blocks,
the coarse basis function with linear boundary condition fails to give a robust bound for the condition number. The coarse basis functions with oscillatory boundary condition, on the other hand, yield a robust method.  First, for each $K\in\mathcal{T}_H,$ let $\rho\geq 1$ be an arbitrary constant, define the set $$K(\rho):=\{x\in K, \alpha(x)>\rho\}.$$ Since $\alpha(x)$ is piecewise constant with respect to $\mathcal{T}_h,$ $K(\rho)$ is a union of fine grid elements.
Let the region $$K(\rho)=K^I(\rho)\cup K^B(\rho)$$ be
 associated with each $K\in \mathcal{T}_H,$ where the set $K^B(\rho)$ contains the components of $K(\rho)$ whose closure touches $\partial K$ and $K^I(\rho)$ contains all the interior components of $K(\rho).$ The term $\varepsilon(\rho, K),$ representing the distance between $K^I(\rho)$ and $K^B(\rho),$ be defined as
 $$\varepsilon(\rho, K)={\rm dist}(K^I(\rho), K^B(\rho)).$$\\
 $\mathbf{Assumption \,\,3.1.}$
 (1)\,\,$K^I(\rho)$ and $K^B(\rho)$ should be well-separated, i.e.,
 $$\varepsilon(\rho, K)\geq h.$$\\
 (2)\,\, $K^B(\rho)$ can be written as a union $K^B(\rho)=\sum_{l=1}^L K_l^B(\rho)$, where the components $K_l^B(\rho)$ are simply connected and pairwise disjoint, and $\alpha(x)$ is a constant on the closure of each component, i.e., $$\alpha(x)=\alpha_l,\,\,{\rm for \,\,all} \,\,x\in\overline{K_l^B(\rho)}.$$\\
 (3)\,\,Let $\Gamma_l^B(\rho)=K_l^B(\rho)\cap\partial K$ be the boundary part of $\partial K$ which locates in $K_l^B(\rho),$ and that $\Gamma^B(\rho)=\sum_{l=1}^L\Gamma_l^B(\rho).$ For all $K\in\mathcal{T}_H,$ let $E$ be any edge of $\partial K,$ we require that
 $$|E\backslash \Gamma^B(\rho)|\gtrsim H_K,$$
 which means that the high conductivity field does not cover too much of $\partial K,\,\,\forall K\in \mathcal{T}_H.$\\
\indent Note that, our assumptions are weaker than those \cite{I.P.R.}, we do not need the coefficient $\alpha(x)$ to be
continuous across the coarse grid boundaries. Next we give the explicit bound for
the indicator $\gamma(\alpha)$ in Theorem \ref{thm:gamma} with a different proof than in \cite{I.P.R.}.
\begin{thm}
Let Assumption 3.1 hold true for each $K\in \mathcal{T}_H,$ then an upper bound for the indicator $\gamma(\alpha),$ with $\psi_{p, \partial K}$ being the oscillatory boundary condition, is given by $$\gamma(\alpha)\lesssim \rho\frac{H}{h}.$$\label{thm:gamma}
\end{thm}
\begin{proof}
 The key idea is to partition $K(\rho)$ into two parts $K_1(\rho)$ and $K_2(\rho),$  c.f., Figure 1, build a special function $\theta\in S^h(K)$ whose bound can be estimated.
\begin{figure}[!h]
  \centering
  \includegraphics[width=7cm, height=5cm]{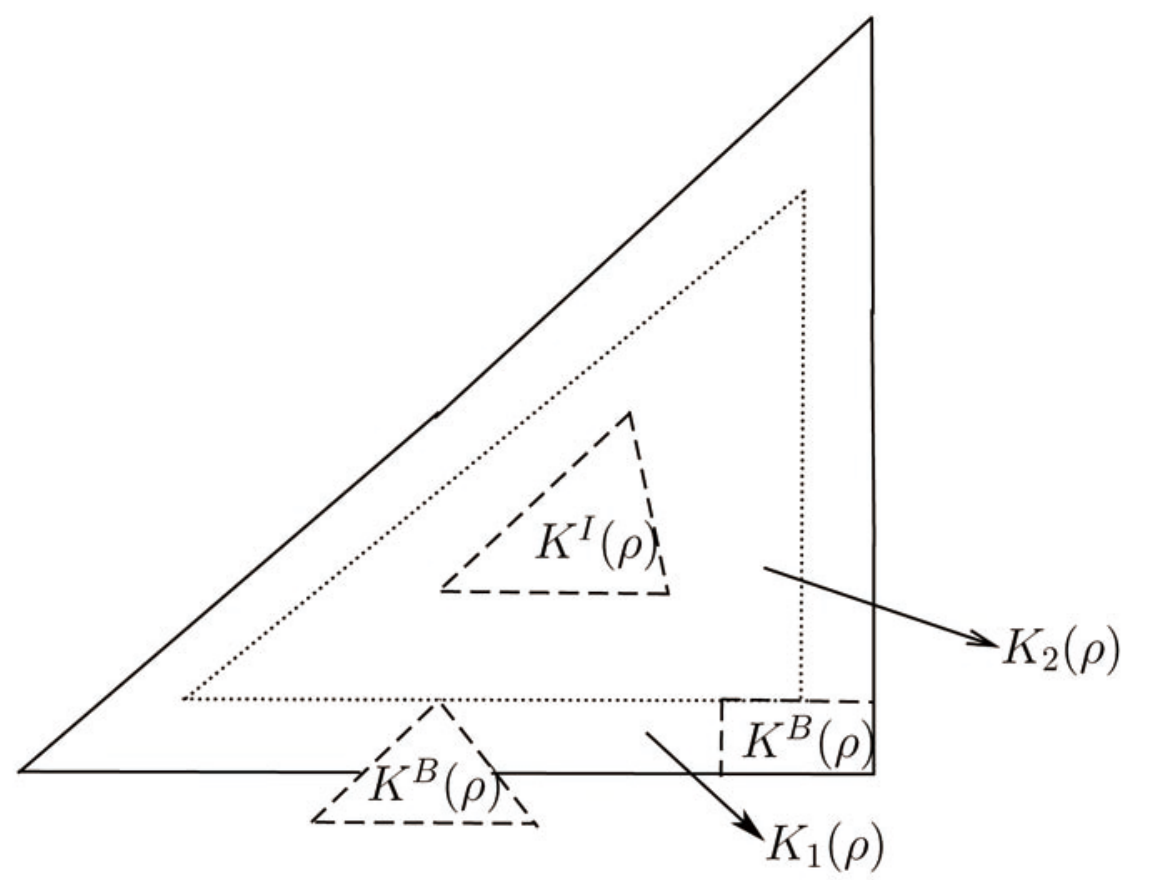}
  \caption{The areas surrounded by dashed lines are $K^B(\rho)$ and $K^I(\rho)$ respectively. $K_1(\rho)$ and $K_2(\rho)$ are separated by dotted lines.}
\end{figure}

Let $K\in \mathcal{T}_H, $ and $\psi_{p, \partial K}$ be the oscillatory boundary condition, $\phi_{p, K}$ is the multiscale basis function which is built through the discrete harmonic extension, c.f., (\ref{eq:discrete}). Accordingly,
$$|\phi_{p,K}|_{1,\alpha,K}\leq|\theta|_{1,\alpha,K}, \,\,{\rm for\,\, all} \,\,\theta\in S^h(K) {\rm\,\, which\,\, satisfies }\,\,\theta|_{\partial K}=\phi_{p,K}|_{\partial K}.$$
\indent \,\,We only need to construct a function $\theta$ for which we can estimate its $|\cdot|_{1,\alpha,K}$ norm. We define the function $\theta$ explicitly by its values at the nodes  of $\mathcal{N}^h(\overline{K}),$
where $\mathcal{N}^h(\overline{K})$ denotes the set of fine mesh vertexes in $\overline{K}.$
Obviously, the function $\theta$ is contained in $S^h(K).$
We begin by constructing $\theta$ on each $K^B(\rho). $ By Assumption 3.1, $K^B(\rho)=\sum_{l=1}^L K_l^B(\rho)$ and that $\Gamma^B(\rho)=\sum_{l=1}^L\Gamma_l^B(\rho)$ with $\Gamma_l^B=K_l^B(\rho)\cap \partial K$ being the part of $\partial K$ locating in $K_l^B(\rho).$ Next, we define the values of $\theta$ on $\Gamma_l^B(\rho),$ then extend to $K_l^B(\rho),$ $\forall\,\, 1\leq l\leq L.$
For simplicity, we assume that $\Gamma_l^B$ lies only in the interior of $\partial K,$ i.e., $\overline{\Gamma_l^B}$ does not touch any vertex of $K.$  For a fixed $l,$ we may choose local coordinate system $(x_1, x_2)$ and some $b_1, b_2>0$ such that $K\subset \{ (x_1, x_2):x_2\geq0\}, \,\,\Gamma_l^B=\{(x_1,0): x_1\in [b_1, b_2]\}.$ Define $\theta(x_1, x_2)=\psi(x_1, 0)$ for all $(x_1, x_2)\in K_l^B(\rho).$ After defining all the values of $\theta$ on $K^B(\rho),$ its values at the nodes of $K\backslash \overline{K^B(\rho)}$ for which we set $0$.\\
\indent \,\, Define $$K_1(\rho)=\{x\in K: {\rm dist}(x, \partial K)\leq d\},$$ where $d=\max_{x\in\overline{K^B(\rho)}}{\rm dist}(x, \partial K).$ Similarly,
$K_2(\rho)=K\backslash \overline{K_1(\rho)}$ denote the area of $K$ which is not contained in $K_1(\rho).$ Clearly, $K^B(\rho)\subset K_1(\rho)$ and $K^I(\rho)\subset K_2(\rho)$ hold by their definitions.\\
\indent \,\,We give an upper bound for $|\theta|_{1,\alpha,K}^2.$ We note,
$$|\theta|_{1,\alpha,K}^2\leq|\theta|_{1,\alpha,K_1(\rho)}^2+|\theta|_{1,\alpha,K_2(\rho)}^2=I_1+I_2.$$
\indent \,\,For the term $I_1,$ we have $$I_1=|\theta|_{1,\alpha,K_1(\rho)}^2=|\theta|_{1,\alpha,K^B(\rho)}^2+|\theta|_{1,\alpha,K_1(\rho)\backslash K^B(\rho)}^2=I_{11}+I_{12}.$$
From (\ref{eq:ebf}) and the nodal values defined above we can see that $(\frac{\partial\theta}{\partial x_1})^2+(\frac{\partial\theta}{\partial x_2})^2=\alpha^{-2}/\big{(}\int_\Upsilon(\alpha^\Upsilon)^{-1}ds\big{)}^2$ on each $\tau\subset K^B(\rho),$ where $\Upsilon\cap K^B(\rho)\neq\emptyset$ is the coarse edge of $K$. Hence
\begin{align*}
\begin{split}
I_{11}=\sum_{\tau\in K^B(\rho)}\int_\tau\alpha((\frac{\partial\theta}{\partial x})^2+(\frac{\partial\theta}{\partial y}))^2dx_1dx_2=\sum_{\tau\in K^B(\rho)}\int_\tau\alpha^{-1}/\big{(}\int_\Upsilon(\alpha^\Upsilon)^{-1}ds\big{)}^2dx_1dx_2.
\end{split}
\end{align*}
\indent Let $\Upsilon_1=\Upsilon\cap K^B(\rho)$ and $\Upsilon_2=\Upsilon\backslash K^B(\rho),$ since $(\alpha^\Upsilon)^{-1}\geq\rho^{-1}$ on $\Upsilon_2,$ it follows from Assumption 3.1 (3) that
\begin{align*}
\int_\Upsilon(\alpha^\Upsilon)^{-1}ds=\int_{\Upsilon_1}(\alpha^\Upsilon)^{-1}ds+\int_{\Upsilon_2}(\alpha^\Upsilon)^{-1}ds>\int_{\Upsilon_2}(\alpha^\Upsilon)^{-1}ds\gtrsim \rho^{-1}|\Upsilon_2|\gtrsim \rho^{-1}H_K,
\end{align*}
which implies that
\begin{align*}
I_{11}=\sum_{\tau\in K^B(\rho)}\int_\tau\alpha^{-1}/\big{(}\int_\Upsilon(\alpha^\Upsilon)^{-1}ds\big{)}^2dx_1dx_2\lesssim\sum_{\tau\in K^B(\rho)}\rho^{-1}|\tau|/\rho^{-2}H_K^2\lesssim \rho.\label{eq:estI11}
\end{align*}
\indent \,\,For the term $I_{12},$ let $K_3(\rho)=K_1(\rho)\backslash \overline{K^B(\rho)},$ and
$$K_{3}(\rho_h)=\{\tau\subset K_3(\rho):\overline{\tau}\cap\partial K_3(\rho)\neq \emptyset\}$$
be the boundary layer of $K_3(\rho)$ with width $h.$  Note that, by definition, we have $\nabla \theta=0$ on
$K_3(\rho)\backslash K_{3}(\rho_h)=\{\tau\subset K_3(\rho):\overline{\tau}\cap\partial K_3(\rho)= \emptyset\},$ and since $\alpha\leq \rho$ on $K_3(\rho),$ we get
\begin{align*}
\begin{split}
I_{12}=&|\theta|_{1,\alpha,K_3(\rho)}^2=\sum_{\tau\in K_3(\rho)}|\theta|_{1,\alpha,\tau}^2 =\sum_{\tau\in K_3(\rho_h)}|\theta|_{1,\alpha,\tau}^2 \lesssim\rho\sum_{\tau\in K_3(\rho_h)}|\theta|_{1,\tau}^2\lesssim \rho\frac{H}{h},
\end{split}
\end{align*}
where in the last inequality we have used the fact that $|\theta|_{1,\tau}^2\lesssim 1$ since all the nodal values of $\theta$ defined on $\tau$ are between $0$ and $1$.\\
\indent \,\,Finally, for the term $I_{2},$ define
$$K_2^B(\rho)=\{\tau\subset K_2(\rho):\\ \overline{\tau}\cap\partial K_2(\rho)\neq\emptyset\},$$
which is the boundary layer of width $h$ of $K_2(\rho),$ and
$$K_2^I(\rho)=K_2(\rho)\backslash K_2^B(\rho),$$ which is the interior part of $K_2(\rho).$\\
\indent By the definition of $\theta$, we have $\theta(x)=0$ on $K_2^I(\rho).$ Since $\alpha\leq \rho$ on $K_2^B(\rho),$ we get
\begin{align*}
\begin{split}
I_2=|\theta|_{1,\alpha,K_2(\rho)}^2
=\sum_{\tau\subset K_2^B(\rho)}|\theta|_{1,\alpha,\tau}^2
\leq\rho\sum_{\tau\subset K_2^B(\rho)}|\theta|_{1,\tau}^2
\leq\rho\frac{H}{h}.
\end{split}
\end{align*}
\indent The theorem is proved by combining the upper bounds for $I_1$ and $I_2$ together.
\end{proof}

\subsection{\rm Upper bound of DG indicator}
In this section, we estimate the term $\beta(\alpha).$ Note that, if the coefficient field is continuous across the coarse grid boundaries, we have $\beta^I(\alpha)=0,$ thus in this case, what we need to focus on is the estimate of the term $\beta^B(\alpha).$ However, in reality, the coefficient field can be discontinuous across the coarse grid boundaries.
In both cases, suppose the coefficient of the conductivity field satisfies $W_e\lesssim \rho,$ where $\rho\geq 1$ is a given number, on the edges of fine mesh triangles which intersect with the coarse grid boundaries. Then the indicator $\beta(\alpha)$ yield a bound which is independent of the high contrast in the coefficients.
\begin{thm}
For all $K\in\mathcal{T}_H,$ $e\subset \partial K,$  let the coefficient field be discontinuous across the coarse grid boundaries. If there exists $\rho>0$ such that $W_e\lesssim \rho,$ then
$$\beta(\alpha)\lesssim \rho\frac{H}{h}.$$\label{thm:beta}
\end{thm}
\begin{proof}
By definition, c.f., (\ref{eq:betaI}), we have
\begin{align*}
\begin{split}
\beta^I(\alpha)=& \max_{\partial K_{mk}\subset\Omega\atop m>k}\max_{x_p\in\mathcal{N}^H(K_m)\atop x_p\in\overline{ \partial K_{mk}}}
 \sum_{e_{jl}\subset\partial K_{mk}\atop j>l}(\int_e\frac{W_e}{h_e}|[\phi_p]|^2ds)\\
 \lesssim&  \max_{\partial K_{mk}\subset\Omega\atop m>k}\max_{x_p\in\mathcal{N}^H(K_m)\atop x_p\in\overline{ \partial K_{mk}}}
 \sum_{e_{jl}\subset\partial K_{mk}\atop j>l}(\int_e\frac{\rho}{h_e}|[\phi_p]|^2ds)\\ \lesssim & \max_{\partial K_{mk}\subset\Omega\atop m>k}\max_{x_p\in\mathcal{N}^H(K_m)\atop x_p\in\overline{ \partial K_{mk}}}
 \sum_{e_{jl}\subset\partial K_{mk}\atop j>l}(\int_e\frac{\rho}{h_e}1^2ds)\lesssim \rho\frac{H}{h},
\end{split}
\end{align*}
using the fact that $|[\phi_p]|\lesssim 1.$ The same result holds for the estimate on $\beta^B(\alpha).$ Hence, since $\beta(\alpha)=\max\big{(}\beta^I(\alpha),\beta^B(\alpha)\big{)},$ the theorem is proved.
\end{proof}%

\section{Numerical Experiments}
\setcounter{equation}{0}
\setcounter{rem}{0}
In this section, we present our numerical results,
where we solve the equation (\ref{eq:eq}, \ref{eq:bc}) with $g=0$ on the square domain $\Omega=[0, 1]^2$.
We run the preconditioned conjugate gradient method
until the $l_2$ norm of the residual is reduced by a factor of $10^6.$\\
\indent In the numerical experiments, subdomains $\{\Omega_i\}_{i=1}^N$ are all square shaped,
and each subdomain consists of two coarse triangles.
In each of our numerical experiments below, we consider the performance of our two methods
and compare with the method in \cite{I.P.R.}.
The results of our method are shown in the columns under the heading "Discontinuous Galerkin",
while the results from \cite{I.P.R.} are shown under the heading "Continuous Galerkin".
We also use different basis functions for the coarse
space, $\phi^L_p,$ $\phi^{MS,L}_p,$ and $\phi^{MS,OSC}_p,$ that is, the piecewise linear basis function,
the multiscale basis function with linear boundary condition,
and the multiscale basis function with oscillatory boundary condition, respectively.
We choose the same penalty parameter $\eta$ for both the fine and coarse bilinear forms.
\begin{figure}[!h]
  \centering
  \includegraphics[width=6cm, height=6cm]{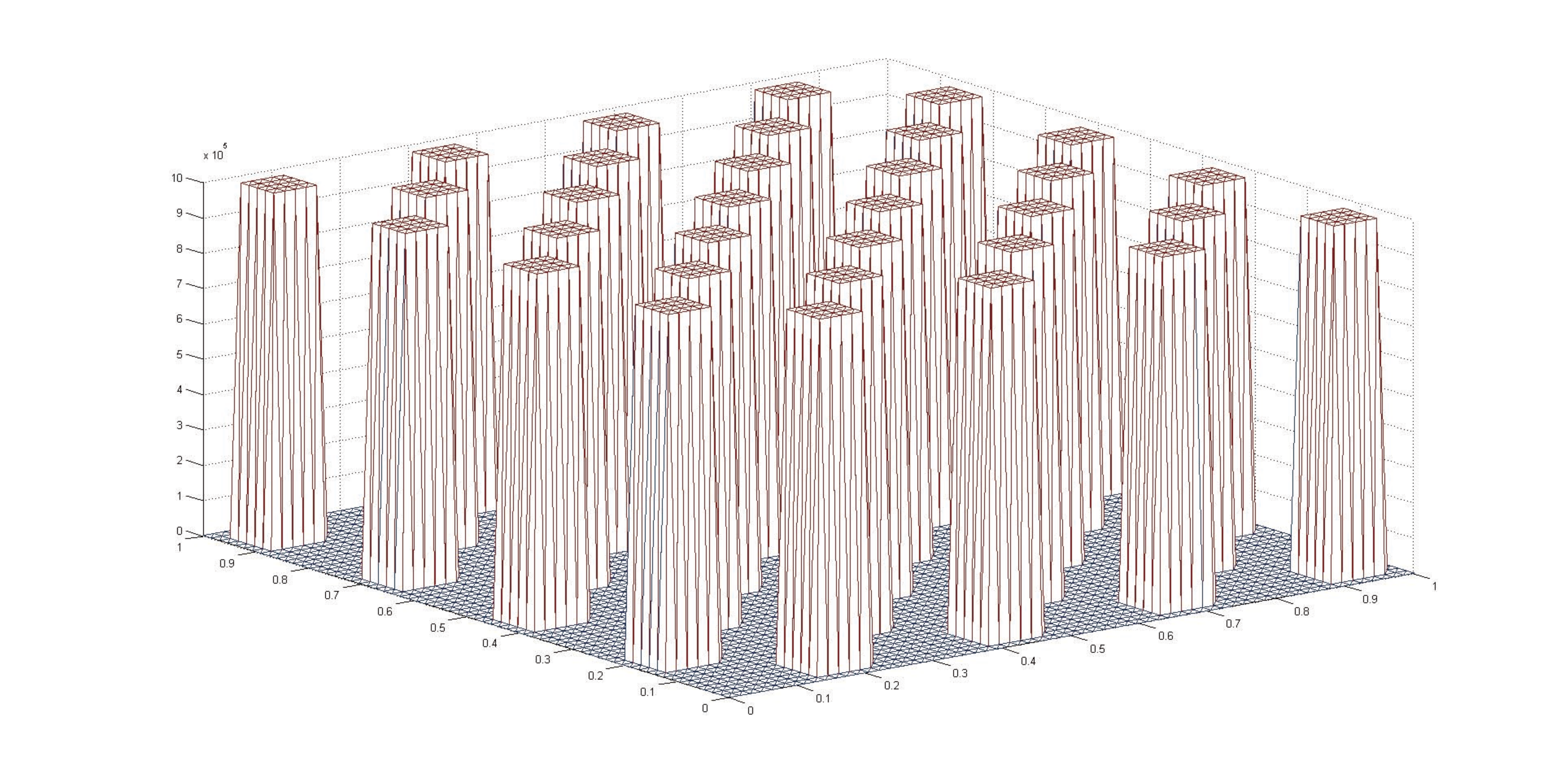}
  \caption{The coefficient $\alpha(x)$ corresponding to the binary domain of Example 4.1.}
\end{figure}
\begin{exm}
{\rm
We begin with our first example here, see Figure 2, where $H\geq 8h$ and $\alpha(x)$ denotes
a 'binary' medium with $\alpha(x)=\widehat{\alpha}$ on a square area inclusion lying in the middle of
each coarse grid element $K\in\mathcal{T}_H,$ and at a distance of $H/8$ both from the horizontal and
the vertical edge of $K,$ and $\alpha(x)=1$ in the rest of the domain. We study the
behavior of the preconditioners as $\widehat{\alpha}\rightarrow\infty.$ \\
\indent For the discontinuous Galerkin formulation, we first consider the nonoverlapping method.
We note that $\max_{K_{mk}\atop m>k}\max_{e\subset \partial K_{mk}}W_e\lesssim 1,$ and that
the DG indicator $\beta(\alpha)=0$ because the coefficient is continuous across the coarse grid boundaries as well as the zero Dirichlet boundary condition.
Taking $\rho=1$ in Theorem \ref{thm:gamma} implies that $\gamma_{DG}^{MS,L}(\alpha)\lesssim \frac{H}{h}.$
It then follows from Theorem \ref{thm:nonoverlapping} that the nonoverlapping method has a bound which is independent of the jumps.
For the overlapping method, we
have $\pi(\alpha)=\delta_i^2\|\alpha|\nabla \chi_i|^2\|_{L^\infty(\Omega)}\lesssim 1,$  thus in this case
the multiscale basis function with linear boundary condition will yield a robust bound.
We note that, in this example, the multiscale basis function with oscillatory boundary condition
is the same as the one with linear boundary condition.\\
\indent The numerical results in Table 1 show that,
both the nonoverlapping and overlapping method with multiscale coarse basis functions with linear boundary conditions
are robust as predicted by the theory.
The overlapping method \cite{I.P.R.} with the same multiscale coarse basis functions
produces almost the same condition number estimates, however, for the linear coarsening,
the results in Table 1 show a loss of robustness of the overlapping Schwarz method as $\widehat{\alpha}$ goes from $10^0$ to $10^6.$ \\

\begin{table}[!h] \label{Table 1}
\caption{Condition number estimates of the Schwarz methods on Example 4.1 with $h=1/128,$ $H=8h,$ $\delta=2h,$ and $\eta=4$.}
\begin{center}
\begin{tabular} {ccccccccc}
  \hline

  \noalign{\smallskip}

  & \multicolumn{3}{c}{~~~~~~Discontinuous Galerkin}&\multicolumn{2}{c}{Continuous Galerkin}\\
 & \multicolumn{1}{c}{Nonoverlapping}&\multicolumn{1}{c}{Overlapping}&&\multicolumn{2}{c}{Overlapping}\\

  \noalign{\smallskip}

$\widehat{\alpha}$\,\,\, &${\rm MS,L}$   &${\rm MS,L}$   &$\gamma_{DG}^{MS,L}(\alpha)$ &${\rm MS,L}$ &${\rm L}$ \\
\hline\noalign{\smallskip}
$10^0$    &26.79                &6.43             &1.00   &5.64    &~~~5.6     \\
$10^2$    &21.61                &6.44             &1.42   &5.79    &~~58.6    \\
$10^4$    &21.49                &6.81             &1.44   &5.80    &358.3    \\
$10^6$    &21.49                &6.81             &1.44   &5.80    &378.7    \\
\hline\noalign{\smallskip}
  \noalign{\smallskip}
\end{tabular}
\end{center}
\end{table}

}
\end{exm}

\indent In our next experiments, we study the behavior with different penalty terms.
As we can see from Table 2,
the condition number estimate for the nonoverlapping method grows linearly with the penalty
parameter $\eta.$ However, it is almost constant for the overlapping method,
which suggests that the condition number bound do not depend on the penalty parameter.
The results in Table 2 is in agreement with Theorem \ref{thm:nonoverlapping} and Theorem \ref{thm:overlapping}.

\begin{table}[!h] \label{Table 2}
\caption{Discontinuous Galerkin formulation on Example 4.1. Condition number estimates of the Schwarz method with $h=1/128,$ $H=8h$ and $\delta=2h({\rm only\,\, for \,\,overlapping}).$ }
\begin{center}
\begin{tabular} {ccccccc}
  \hline

  \noalign{\smallskip}
 &\multicolumn{3}{c}{Nonoverlapping method}&\multicolumn{3}{c}{Overlapping method} \\
  $\widehat{\alpha}$  &  $\eta=5$
  & $\eta=10$
  &$\eta=100$ & $\eta=5$ & $\eta=10$ & $\eta=100$\\
  \noalign{\smallskip}

  \hline\noalign{\smallskip}
$10^0$    &33.14    &64.88     &635.2    &6.47 &6.56  &6.88   \\
$10^2$    &26.55    &51.09     &487.3   &6.73 &6.87  &6.96    \\
$10^4$    &26.39    &50.74    &483.4   &6.83 &6.89  &6.99    \\
$10^6$    &26.40    &50.73     &483.4   &6.83 &6.89  &6.99    \\

  \noalign{\smallskip}

  \hline

\end{tabular}

\end{center}

\end{table}
\begin{rem}
{\rm  Note that, in our numerical experiments, we test the model problem with zero Dirichlet boundary condition, thus all the degrees of freedom are inside domain $\Omega,$ which means that our DG indicator $\beta(\alpha)$ defined in (\ref{eq:beta}) is only for $\partial K_{mk}\subset \Omega$ but not for the coarse grid boundaries on $\partial \Omega.$ Moreover, in Example 4.1, since the coefficient of the conductivity field is continuous across coarse grid boundaries, we have $\beta(\alpha)=0$ due to the fact that the jump of the coarse basis functions corresponding to the same coarse node is equal to 0. From Theorem \ref{thm:nonoverlapping}, we know that the condition number bound for the nonoverlapping case will be $$\eta\max_{\partial K_{mk}\atop m>k} \max_{e_{jl}\subset \partial K_{mk}\atop j>l}W_e\gamma(1)\frac{H}{h}+\gamma(\alpha),$$ taking $\rho=1$ in Theorem \ref{thm:gamma} we have $\gamma(\alpha)\lesssim \frac{H}{h}$, thus theoretically speaking, the condition number bound will grow linearly with the penalty parameter $\eta.$ While for the overlapping case, From Theorem \ref{thm:overlapping}, we know that the condition number bound for the overlapping case will be $$\pi(\alpha)\max(\gamma(1),\beta(1))\max_i\frac{H_i}{\delta_i}+\gamma(\alpha),$$ taking $\rho=1$ in Theorem \ref{thm:gamma} implies $\gamma(\alpha)\lesssim \frac{H}{h}$, thus in this case, our condition number bound will keeps unchanged independent of the penalty parameter $\eta$. Both of the two cases are matched exactly by our numerical experiments in Table 2.

}
\end{rem}
\begin{figure*}[!h]
\centering
\begin{minipage}[t]{0.45\linewidth}
    \centering
    \includegraphics[width=6cm, height=6.5cm]{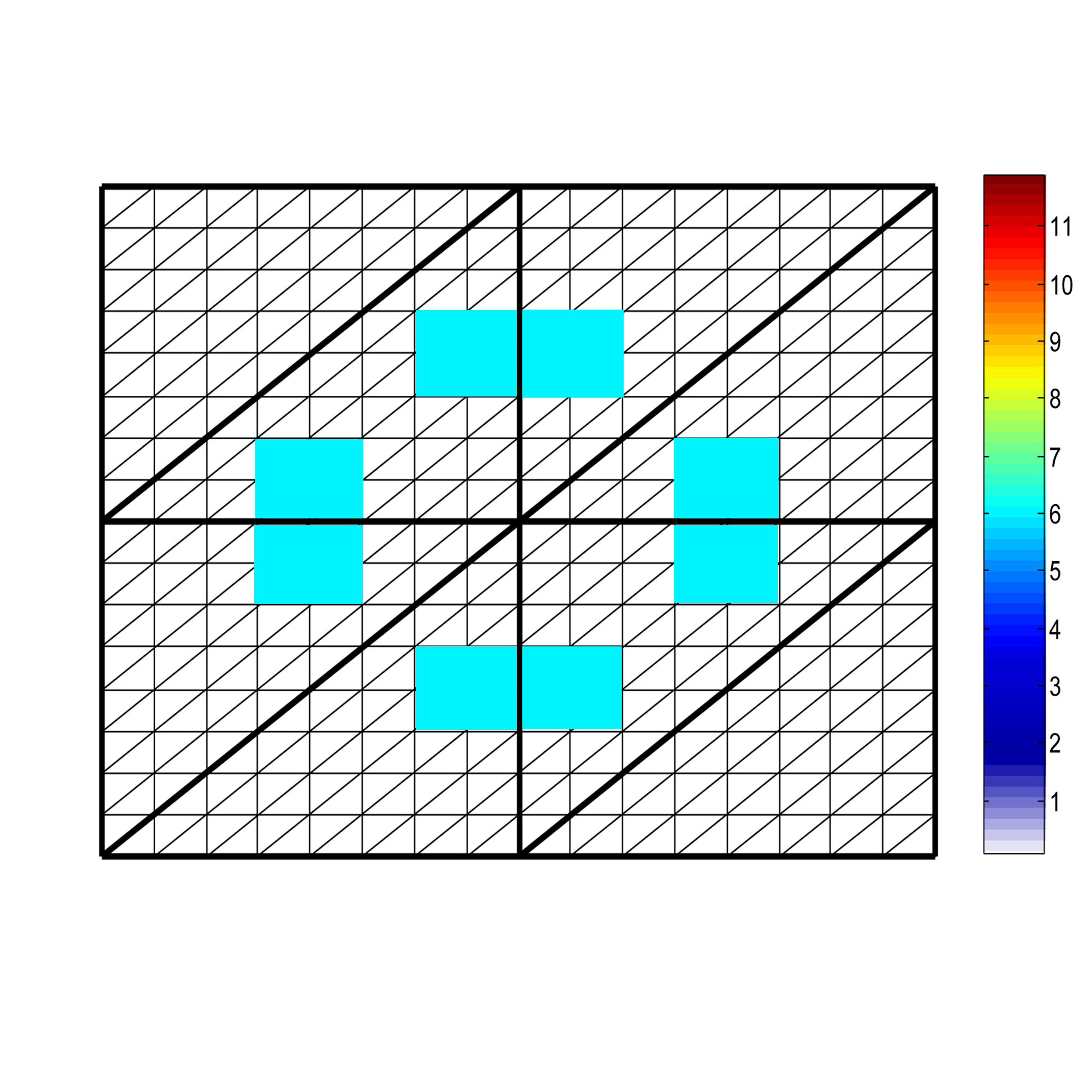}
    \caption{Example 4.2 with $\alpha(x)$ crossing the coarse grid boundaries, $\alpha(x)=\widehat{\alpha}$ on the channels and $\alpha(x)=1$ otherwise.}
\end{minipage}
\hspace{3ex}
\begin{minipage}[t]{0.45\linewidth}
    \centering
    \includegraphics[width=6cm, height=6.5cm]{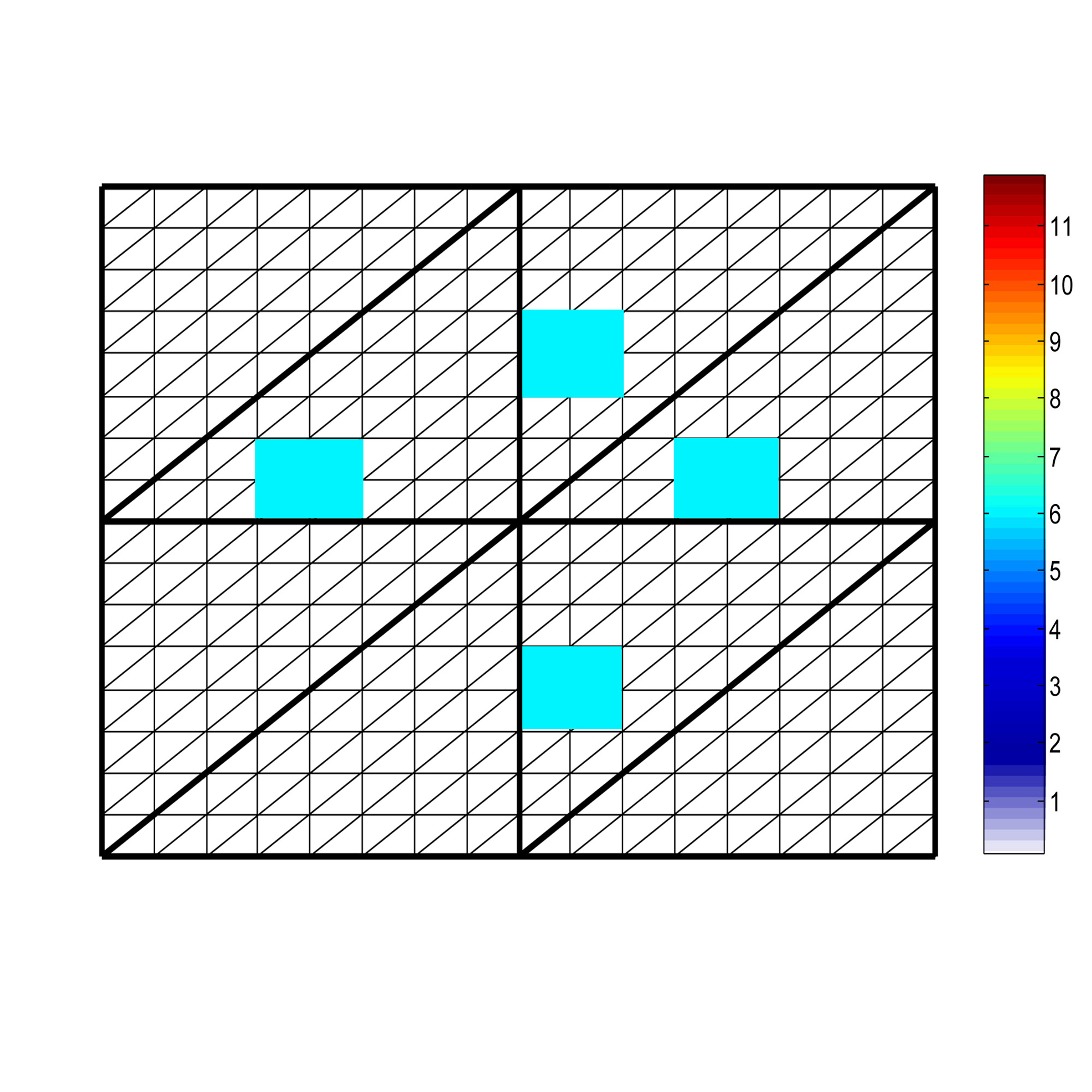}
     \caption{Example 4.2 with $\alpha(x)$ touching the coarse grid boundaries, $\alpha(x)=\widehat{\alpha}$ on the channels and $\alpha(x)=1$ otherwise.}
\end{minipage}
\end{figure*}

\begin{exm}
{\rm
In this example, the high conductivity field is no longer contained inside,
it touches the coarse grid boundaries, see Figure 3 and Figure 4.
In Figure 3, The conductivity field $\alpha(x)$ corresponds to a binary medium with background $\alpha(x)$ being equal
to one, and $\alpha=\widehat{\alpha}$ on the channels with an area equal to $2h\times4h$ each.
In Figure 4, the high conductivity channels with diameter $2h\times 2h$ are located only on one side of the coarse grid boundaries.\\
\indent We first compare the behavior of our nonoverlapping method on the two cases shown in Figure 3 and Figure 4.
For the the conductivity field given in Figure 3, $\gamma(\alpha)\lesssim \frac{H}{h}$
since $\max_{K}\max_{e\subset \partial K}W_e=\widehat{\alpha}$,
$\gamma(\alpha)\lesssim \frac{H}{h}$ (taking $\rho=1$ in Theorem \ref{thm:gamma}),
and $\beta(\alpha)=0$. Thus Theorem  \ref{thm:nonoverlapping} predicts that the condition number bound will
grow as $\widehat{\alpha}\rightarrow \infty.$
However, for the conductivity field in Figure 4, we
have $\max_{K}\max_{e\subset \partial K}W_e\lesssim 1,$ $\gamma(\alpha)\lesssim \frac{H}{h}$ (taking $\rho=1$ in Theorem \ref{thm:gamma}),
and $\beta(\alpha)\lesssim \frac{H}{h}$ (taking $\rho=1$ in Theorem \ref{thm:beta}).
Theorem \ref{thm:nonoverlapping} predicts that the bound is robust in this case. For the overlapping method,
it is easy to see
that $\pi(\alpha)\lesssim \max_i\delta_i^2\|\alpha\nabla |\chi_i|^2\|_{L_\infty(\Omega)}\lesssim \delta_i^2/h^2\approx1.$ It
follows from Theorem \ref{thm:overlapping} that the condition number bound will be independent of the jumps
for both the conductivity field in Figure 3 and in Figure 4.
The numerical results in Table 3 confirm these results.

\begin{table}[!h] \label{Table 5}
\caption{Discontinuous Galerkin formulation on Example 4.2. Condition number estimates
for the Schwarz method with $h=1/128,$ $H=8h,$ $\delta=4h,$ and the penalty parameter $\eta=4 $. }
\begin{center}
\begin{tabular} {cccccc}
  \hline

  \noalign{\smallskip}

  &  \multicolumn{4}{c}{~Nonoverlapping~ Overlapping}  \\

  \noalign{\smallskip}

$\widehat{\alpha}$\,\,\, &  ${\rm Fig\,\,3}$& ${\rm Fig\,\,4}$ &  ${\rm Fig\,\,3}$&  ${\rm Fig\,\,4}$\\
  \noalign{\smallskip}

  \hline\noalign{\smallskip}

$10^0$         &2.68e+1 &26.79     &6.23     &6.24 \\
$10^2$         &5.00e+2 &28.91      &7.17    &7.15  \\
$10^4$         &4.76e+4 &28.95  &7.23   &7.17 \\
$10^6$         &4.76e+6  &28.96  &7.23  &7.17 \\
\hline\noalign{\smallskip}

\end{tabular}

\end{center}

\end{table}
}
\end{exm}

\begin{figure}[!h]
  \centering
  \includegraphics[width=5cm, height=5cm]{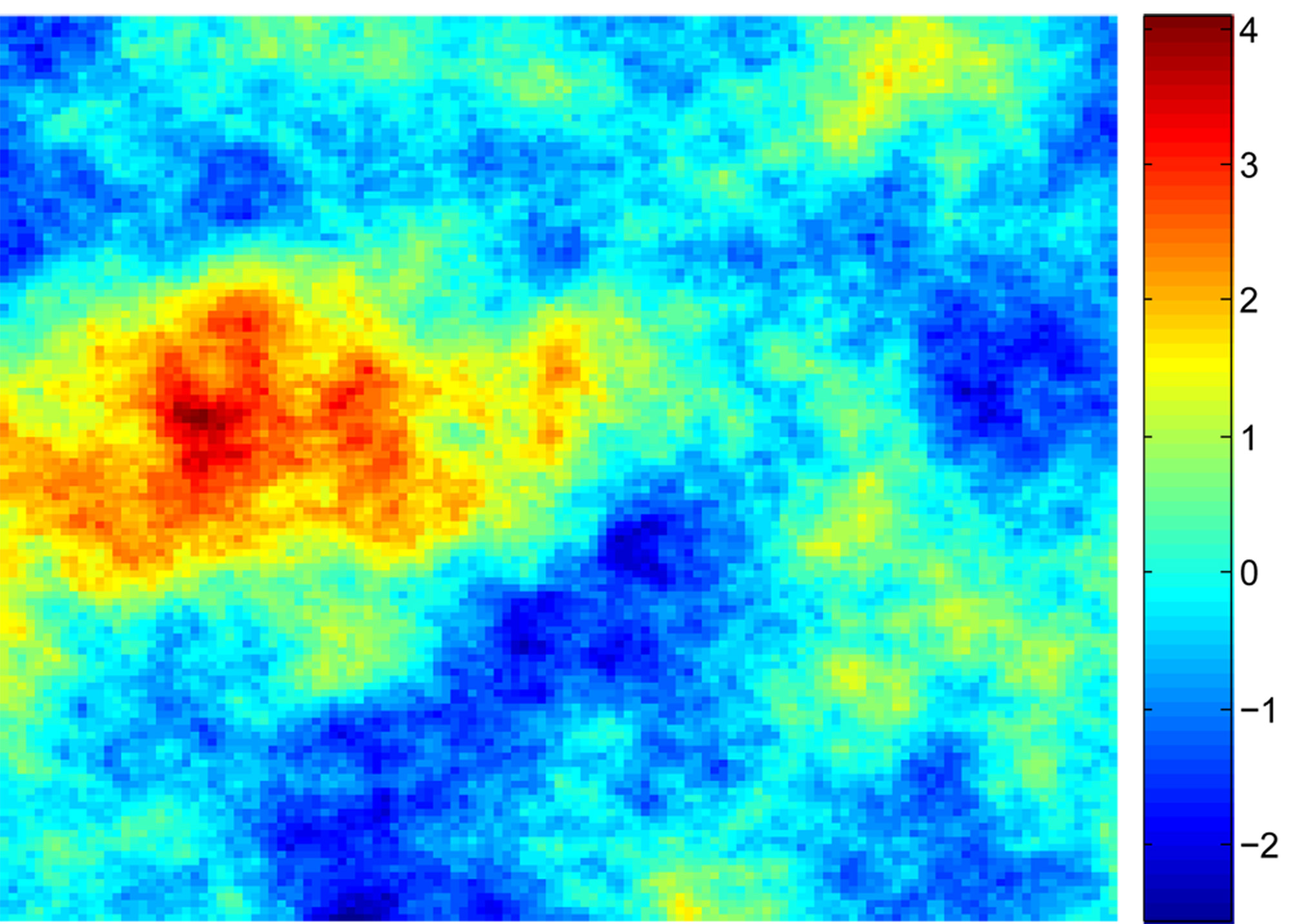}
  \caption{$256\times256$ Gaussian random field with covariance parameter $\theta=50.$   }
\end{figure}

\begin{table}[!h] \label{Table 7}
\caption{Condition numbers (iteration numbers) of Schwarz methods on Example 4.3 with $h=1/128,$ $H=8h,$ $\delta=2h,$ and $\eta=4$.}
\begin{center}
\begin{tabular} {lllllll}
  \hline

  \noalign{\smallskip}

  & \multicolumn{2}{c}{Discontinuous Galerkin }&\multicolumn{1}{c}{Continuous Galerkin}\\
 & \multicolumn{1}{c}{Nonoverlapping}&\multicolumn{1}{c}{Overlapping}&\multicolumn{1}{c}{Overlapping}\\

  \noalign{\smallskip}

${\rm High}\,\, {\rm Contrast}$\,\,\, &${\rm MS,OSC}$   &${\rm MS,OSC}$   &$~~~~~~{\rm MS,OSC}$     \\
\hline\noalign{\smallskip}
8.55e+3     & $30.72~(52)$         &~\,7.20~(16)       &~~~~~~\,~6.65~(15)       \\
1.75e+5     & $32.68~(57)$         &~\,7.85~(18)        &~~~~~~\,~7.12~(16)       \\
7.32e+7     & $36.40~(67)$         &~\,9.86~(21)         &~~~~~~\,~8.45~(19)     \\
1.49e+9     & $37.90~(73)$         &11.27~(22)       &~~~~~~\,~9.38~(20)     \\
6.02e+11    & $45.94~(86)$         &14.77~(25)    &~~~~~~11.73~(22)    \\
1.28e+13    & $51.61~(96)$         &16.80~(26)    &~~~~~~13.28~(25)       \\
\hline\noalign{\smallskip}
  \noalign{\smallskip}
\end{tabular}
\end{center}
\end{table}

\begin{exm}
{\rm
[Gaussian random field] In this example, we test our method on a more realistic model.
The coefficient $\alpha$ is a realisation of a log-normal random field (see Figure 5), i.e., $\log\alpha(x)$ is
a realisation of a homogeneous isotropic Gaussian random field with spherical covariance function, mean 0.
This is a commonly studied model in the multiscale area in many literatures \cite{I.P.R.,C.G.S.S.}.
There are some evidences from field data that this gives a reasonable
presentation of reality in certain cases \cite{Gelhar,H.K.}.
There are many good ways to generate such random fields, we simply follow
the way in \cite{K.K.} which used FFT (Fast Fourier Transformation) method \cite{FR.DM.,R.G.S.W.}.
The spherical covariance function has a parameter $\theta$, a bigger $\theta$ increases the correlation
of the random field.\\
\indent In the numerical experiments, we compared the condition number
estimates (iteration numbers) of both the nonoverlapping and overlapping domain decomposition method
with the method \cite{I.P.R.}.
}
\end{exm}

\section{Conclusions}

In this paper, we present two two level additive Schwarz domain decomposition methods
for the multiscale second order elliptic problem.
The work here is an extension of related works, c.f., \cite{D.S.,I.P.R.}, to a discontinuous Galerkin formulation, c.f., \cite{Z.X.S.}.
For the nonoverlapping case,
we show that if the conductivity field do not cross the subdomain boundaries, our nonoverlapping method is robust.
However, in some cases, when the conductivity field cross the subdomain boundaries, we show in Theorem \ref{thm:nonoverlapping}
that our condition number estimate will be dependent on $W_e$, i.e., the weighted average of the
coefficient over the edge $e.$  It is now the topic of further investigation to see if we can get rid of this dependence
on $W_e$. For the overlapping method, as seen in Theorem
\ref{thm:overlapping}, we show that our method is robust under Assumption 3.1. We allow the discontinuity of the coefficients across the coarse grid
boundaries.\\

~~\\

{\bf Acknowledegments} \,\,The authors thank Professor Maksymilian Dryja for a fruitful discussion on the paper and, in particular, for his valuable comments on the discontinnuous Galerkin method.

\renewcommand{\refname}{

\end{document}